\documentclass{article}
\usepackage[T1]{fontenc}
\usepackage[english]{babel}
\usepackage[margin=1.in]{geometry}
\usepackage[utf8]{inputenc}
\usepackage{lmodern}
\usepackage{microtype}

\usepackage{amsmath,amssymb,amsfonts,amsthm}
\usepackage{mathtools}
\usepackage{mathrsfs} 
\usepackage[square,comma,numbers]{natbib}
\usepackage[colorlinks=true, allcolors=blue]{hyperref}
\usepackage{graphicx}
\usepackage{enumerate}
\usepackage[font=small]{caption}
\usepackage{subcaption}
\usepackage{xcolor}
\usepackage{tikz}

\numberwithin{equation}{section}

\def\cH{{\mathcal H}}

\def\cJ{{\mathcal J}}

\def\cR{{\mathcal R}}
\def\cS{{\mathcal S}}

\def\cT{{\mathcal T}}

\def\cX{{\mathcal X}}
\def\cY{{\mathcal Y}}

\def\E{\mathbb{E}}

\def\N{\mathbb{N}}

\def\P{\mathbb{P}}
\def\R{\mathbb{R}}

\def\Z{\mathbb{Z}}

\def\sA{\mathscr{A}}

\def\sH{\mathscr{H}}
\def\sL{\mathscr{L}}

\def\sX{\mathscr{X}}

\def\dom{\mathrm{dom}}
\def\Lip{\mathrm{Lip}}

\renewcommand{\d}{\mathrm{d}}
\newtheorem*{Def*}{Definition}
\newtheorem*{Thm*}{Theorem}
\newtheorem*{Cor*}{Corollary}
\newtheorem*{Rmk*}{Remark}
\newtheorem*{Lem*}{Lemma}
\newtheorem*{Prop*}{Proposition}
\newtheorem*{Asm*}{Assumption}

\newtheorem{Def}{Definition}[section]
\newtheorem{Thm}[Def]{Theorem}
\newtheorem{Cor}[Def]{Corollary}
\newtheorem{Rmk}[Def]{Remark}
\newtheorem{Lem}[Def]{Lemma}
\newtheorem{Prop}[Def]{Proposition}
\newtheorem{Asm}[Def]{Assumption}
\newtheorem{Ex}[Def]{Example}

\def\be{\begin{equation}}
\def\ee{\end{equation}}

\usepackage{fancyhdr}
\begin{document}
\title{Markov Perfect Equilibria in \\ Discrete Finite-Player and Mean-Field Games\footnote{Research of the first two authors is partially supported 
by the National Science Foundation grant DMS 2406762, and
research of the third author is partially supported by the Simons Foundation Award \#949877. }}

\author{Felix H\"{o}fer\footnote{Department of Operations Research and Financial Engineering, Princeton University, Princeton, NJ 08540, USA}
\and H. Mete Soner\footnotemark[2]
\and Atilla Y\i lmaz\footnote{Department of Mathematics, Temple University, Philadelphia, PA 19122, USA}
}

\date{}

\maketitle

\begin{abstract}
\noindent
We study dynamic finite-player and mean-field  stochastic games within the framework of Markov perfect equilibria (MPE). Our focus is on discrete time and space structures without monotonicity. Unlike their continuous-time analogues, discrete-time finite-player games generally do not admit unique MPE. However, we show that uniqueness is remarkably recovered when the time steps are sufficiently small, and we provide examples  demonstrating the necessity of this assumption. This result, established without relying on any monotonicity conditions, underscores the importance of inertia in dynamic games. In both the finite-player and mean-field settings, we show that MPE correspond to solutions of the Nash-Lasry-Lions equation, which is known as the master equation in the mean-field case. We exploit this connection to establish the convergence of discrete-time finite-player games to their mean-field counterpart in short time. Finally, we prove the convergence of finite-player games to their continuous-time version on every time horizon. 
\end{abstract}


\section{Introduction}

Mean-field games (MFGs) were independently introduced by Lasry and Lions \cite{lasry_jeux_2006-1, lasry_jeux_2006, lasry_mean_2007} and Huang, Caines and Malhamé \cite{huang_individual_2003,huang_invariance_2007, huang_large-population_2007, huang_nash_2007} as a framework to approximate large symmetric games in which the interaction is through the empirical distribution of agents. Since their introduction, they have been widely applied, ranging from heterogeneous agent models in economics to models of jet-lag recovery in neuroscience, see for example \cite{carmona2020applications,carmona2018probabilistic,carmona2020jet}.

In recent years, considerable effort has been devoted to identifying conditions that ensure uniqueness of mean-field game equilibria beyond the Lasry-Lions monotone regime, as well as to developing selection criteria in cases where uniqueness fails \cite{bardi2019non,cecchin2019convergence,delarue2019restoring,hajek2019non,bayraktar2020non,delarue2020selection,dianetti2021submodular,cecchin2022selection,carmona2023synchronization,graber2023monotonicity, hofer2025synchronization}. The starting point of our present investigation is the observation by Gomes, Mohr and Souza \cite{GMS1} that continuous-time, finite-state games with a \emph{finite number} of players admit unique Markov perfect equilibria (MPE)\footnote{Markov perfect equilibria are sometimes referred to as \emph{Markovian Nash equilibria} (e.g., by Carmona and Delarue \cite{carmona2018probabilistic}) or \emph{feedback equilibria} (e.g., by Ba\c{s}ar and Olsder \cite{BO}). The term \emph{Markov perfect equilibria} (see Fudenberg and Tirole \cite{fudenberg1991game}) has become classical in the theory of stochastic games, and we adopt this convention.}, while open-loop equilibria are generally not unique. This result holds under mild regularity conditions and is established by characterizing MPE as solutions to a system of ordinary differential equations, with uniqueness guaranteed by the classical Lipschitz theory. 

In this work, we are interested in the existence and uniqueness of MPE in discrete-time, finite-state games, both in the finite-player and mean-field settings. Stochastic finite-player games in discrete time are classical in game theory \cite{BO}, and starting from Gomes, Mohr and Souza \cite{GMS2}, the mean-field version has been extensively studied theoretically \cite{SBR,doncel2019discrete,bonnans2023discrete}, in the context of learning procedures \cite{yang2017learning, guo2019learning, perrin2021mean} and as approximations for continuous-time MFGs \cite{achdou2013mean, hadikhanloo2019finite, garcia2025mean}. Related solution concepts for approximating Markov perfect equilibria in games with a large number of players have been proposed in economic theory, see Adlakha and Johari \cite{adlakha2013mean}, Weintraub, Benkard and Van Roy \cite{weintraub2008markov} and the references therein. In the finite-player (resp.\ mean-field) case, we characterize symmetric MPE by an equation which we call the \emph{finite-player (resp.\ mean-field) Nash-Lasry-Lions equation (NLL)}. In the mean-field limit, this equation is known as the \emph{master equation}. While classical and various weak solutions to the NLL equation in the continuous-time setting have been studied extensively, see for example \cite{cardaliaguet2019master,  cecchin2022weak,chassagneux2022probabilistic, gangbo2022global,cardaliaguet2022monotone, mou2024mean, mou2024minimal}, NLL equations in discrete time have received less attention. 

In the finite-player problem, we show that while solutions to the finite-player NLL equation are generally not unique, uniqueness is remarkably restored for all sufficiently small time steps without
relying on any monotonicity conditions. Intuitively, the smallness of the time step renders certain transitions suboptimal, thereby underscoring the importance of inertia in dynamic games.
Additionally, as the time step goes to zero, we prove that the solution of the discrete-time $N$-NLL equation converges to its continuous-time counterpart. Convergence of time discretizations in the mean-field problem have been well-studied \cite{achdou2013mean, hadikhanloo2019finite, dianetti2024pasting, garcia2025mean}. 

In the mean-field setting, MPE correspond to solutions of the mean-field Nash-Lasry-Lions equation. This equation is well-defined as it does not involve derivatives, and any solution induces a solution to the mean-field game for any initial condition. We present an explicit example of a one-step game where uniqueness is lost when the time step is large, and we establish existence and uniqueness in a general game when the time horizon and time step are sufficiently small. In continuous time, it is well-known that the NLL equation may develop singularities in finite time, see \cite{cecchin2019convergence, bayraktar2020non} for explicit examples, and, hence, classical solutions only exist in short time in the absence of monotonicity conditions \cite{GMS1, gangbo2015existence, cardaliaguet2022splitting, ambrose2023well}. Our result can be seen as the discrete-time analogue of these facts. Additionally, in the short-time regime, we prove that solutions to the finite-player NLL equation converge to their mean-field counterpart in discrete time---analogous to established results in the continuous-time setting, see for example \cite{GMS1}. 

To summarize, our main contributions are threefold. First, we clarify the relationship between symmetric Markov perfect equilibria and the Nash-Lasry-Lions equation in both the finite-player and the mean-field problem. Second, we study existence and uniqueness of MPE in the absence of monotonicity assumptions. While MPE in finite-player problems exist under mild regularity conditions, they become unique for sufficiently small time steps. In this case, our results contribute to the classical theory of dynamic finite-player games, see Ba\c{s}ar and Olsder \cite{BO}. In the mean-field problem, we prove existence and uniqueness of mean-field MPE for short time horizons. Finally, we establish two convergence results: (1) the convergence of the discrete-time finite-player MPE to its continuous-time counterpart as the time discretizations go to zero; and (2) for small time steps and short time horizons, the convergence of finite-player MPE to mean-field MPE.

We continue to define the problem and to present our results, which we compare to the existing literature in Section \ref{ssec:literature}.

\vspace{1em}

\textbf{Notation.} For $\cX:=\{1,\ldots,d\}$, the space $\R^\cX$ is identified with $\R^d$, and we interchangeably use the notation $a(x)$ and $a_x$ to denote the $x$th coordinate of any $a\in\R^d$. The $(d-1)$-dimensional vector $a_{-x}$ is defined by deleting the $x$th entry from $a$. Further, if $\psi:A\times B\to C$ is a function, then $\psi(a):B\to C$ is defined by $\psi(a)(b):=\psi(a,b)$. For $p\in[0,\infty]$, the $p$-norm is denoted by $|\cdot|_p$. For $p=0$, we recall the definition of the $0$-norm: $|a|_0 :=|\{x:a_x\neq0\}|$. For a function $\psi:\R^k\to\R^{k'}$, we define its Lipschitz constant with respect to the $|\cdot|_1$-norm by $\Lip_1(\psi) := \sup_{\tilde \mu\neq\mu} |\psi(\tilde \mu)-\psi(\mu)|_1/|\tilde \mu-\mu|_1$. Finally, probability distributions/vectors are treated as row vectors.

\section{Problems and main results}

\subsection{The general finite-player problem}

We consider a dynamic stochastic game between $N+1$ players or agents indexed by $\{0,\ldots,N\}$ in discrete time over a finite time horizon $\bar \cT=\{0,h,2h, \ldots, Kh\}$. 
Here, $h>0$ is the time step and we denote the time horizon by $T=Kh$. We write $\cT = \bar \cT\setminus\{T\}$. Let us fix a player $0\leq i \leq N$. Player $i$'s state is described by a discrete-time stochastic process $(X^i_t)_{t\in\bar \cT}$ taking values in a finite set $\cX=\{1,\ldots,d\}$.\footnote{The specific choice of $\cX$ is irrelevant. In particular, we do not assume any ordering on $\cX$.} This player uses a feedback policy $ \alpha^i:\cT\times\cX^{N+1}\to\Sigma^{d-1}$ to control the transition probabilities of $X^i_t$. The feedback control $\alpha^i$ is allowed to depend on the current time and position of the $N+1$ players. Here, the set of probability measures on $\cX$ is identified with the $(d-1)$-dimensional probability simplex $\Sigma^{d-1}:=\{p\in[0,1]^d:\sum_x p_x= 1\}$.
Given feedback policies $\alpha=(\alpha^0,\ldots,\alpha^{N})$ and given players' states $(x^0,\ldots,x^N)$ at time $t$, the states of players at time $t+h$ are drawn independently of each other. More precisely, let $X=(X^0,\ldots,X^N)$ be the state process which tracks the positions of all players. Without loss of generality, we let $X$ be the canonical element in the space $\Omega=(\cX^{N+1})^{\bar\cT}$. Its $\P^\alpha$-transition probabilities are given by
$$
\P^\alpha[X_{t+h} = y \,|\, X_t=x] = \prod_{i=0}^N \alpha^i(t,x)(y^i),\quad t\in\cT,\ x,y\in\cX^{N+1}.
$$
Given the controls of the other players $\beta=(\beta^j)_{j\neq i}$, player $i$  minimizes the following expected cost, which is the discretization of the continuous-time problem that we discuss in Section \ref{ssec:literature} below. Inspired by the continuous-time setting, our main structural assumption on the cost functional is that we penalize the \emph{rate} of changing states.  Precisely, at time $t\in\bar\cT$ and position $x\in\cX^{N+1}$ of all players, player $i$ solves
\be\label{eq:N-player-general-value-function}
\begin{split}
v^i(t,x;\beta)&:=\inf_{\alpha^i}\ \cJ^i(t,x,\alpha^i;\beta),\quad \text{where}\\
\cJ^i(t,x,\alpha^i;\beta) &:= \E^{\alpha^i\otimes \beta}\Bigl[
\sum_{t\leq s<T} \ell^i(X_s,\alpha^i(t,X_s)/h)h+g^i(X_T)
\,|\,X_t=x\Bigr],
\end{split}
\ee
$\ell^i: \cX^{N+1}\times [0,\infty)^{d}\to [0,\infty)$ and $g^i:\cX^{N+1}\to [0,\infty)$ are running and terminal cost functions, respectively, and we use the notation
$$
(\alpha^i\otimes \beta)^j:=
\begin{cases}
\alpha^i,&i=j,\\
\beta^j,&i\neq j,
\end{cases}
$$
for any $0\leq i,j\leq N$. $\cJ^i$ is the \emph{cost functional} while $v^i$ is the \emph{value function} of player $i$. 
For future use, for a given strategy profile of all players $\alpha=(\alpha^0,\ldots,\alpha^{N})$, we define the vector of strategies of players other than $i$ by $\alpha^{-i}:=(\alpha^j)_{j\neq i}$.

\begin{Def}[Markov perfect equilibrium]
{\rm
A strategy profile $\alpha^* = (\alpha^{*,0},\ldots, \alpha^{*,N}):\cT\times\cX^{N+1}\to(\Sigma^{d-1})^{N+1}$ is called a \emph{Markov perfect (Nash) equilibrium (MPE)} if
$$
\cJ^i(t,x,\alpha^{*,i};\alpha^{*,-i})\leq \cJ^i(t,x, \alpha^i;\alpha^{*,-i})
$$
for any feedback control $\alpha^i:\cT\times\cX^{N+1}\to\Sigma^{d-1}$, any $(t,x)\in\cT\times\cX^{N+1}$ and any $0\leq i\leq N$. 
}
\end{Def}

We recall some facts about discrete dynamic programming equations in Appendix \ref{app:dynamic-progamming}. Finding MPEs amounts to solving a coupled system of dynamic programming equations. This set of equations is commonly referred to as the \emph{Nash system} and it reads, for $0\leq i\leq N$, $0\leq t<T$ and $x\in\cX^{N+1}$,
\be \tag{Nash system}\label{eq:Nash-System}
\left\{
\begin{split}
v^i(t,x) &= \underset{a^i\in\Sigma^{d-1}}{\inf}\ \sH^i(x,v^i(t+h),a^i;\alpha^{-i}(t,x)),\\
\alpha^i(t,x) &= \underset{a^i\in\Sigma^{d-1}}{\mathrm{arg\,min}} \ \sH^i(x,v^i(t+h),a^i;\alpha^{-i}(t,x)),\\
v^i(T,x) &= g^i(x).
\end{split}
\right.
\ee
Here, for $(x,\varphi,a)\in\cX^{N+1}\times \R^{d(N+1)}\times\Sigma^{d-1}$, we defined player $i$'s \emph{Hamiltonian} by
$$
\sH^i(x,\varphi,a;\alpha^{-i}(t,x))  := \ell^i(x,a/h)h + \sum_{y=(y^0,\ldots,y^N)\in\cX^{N+1}} \varphi(y) a(y^i)\prod_{j\neq i}\alpha^j(t,x)(y^j).
$$ 

In the next section, we consider a game where $(\ell^i,g^i)=(\ell,g)$ for all $0\leq i\leq N$, with some common cost functions $\ell,g$, and show how \eqref{eq:Nash-System} reduces to the Nash-Lasry-Lions equation. We note that the techniques introduced in the next section to prove the existence and uniqueness of the solutions to the Nash-Lasry-Lions equation can also be used to establish the same for \eqref{eq:Nash-System}.

\subsection{The symmetric finite-player problem}

We now consider a game with symmetric players and derive the Nash-Lasry-Lions equation. The assumption of symmetry will allow us to obtain a \emph{mean-field game} in the limit $N\to\infty$.  More precisely, we assume that, for all $0\leq i\leq N$, the cost functions are of \emph{mean-field type}:
\be \label{eq:symmetry}
\ell^i(x,a) = \ell(x^i, \frac1N \sum_{j\neq i}\delta_{\{x^j\}} ,a),\quad g^i(x) = g(x^i,\frac1N \sum_{j\neq i}\delta_{\{x^j\}})
\ee
for some functions $\ell:\sX_N\times[0,\infty)^d\to[0,\infty)$, $g:\sX_N\to[0,\infty)$. Here we set
$$
\Sigma^{d-1}_N:=\{z\in\Sigma^{d-1} : z_x\in\{0,1/N,\ldots,1\}\ \forall x\in\cX\},\quad \sX_N := \cX\times\Sigma^{d-1}_N,
$$
and we identify the set of empirical probability measures on $\cX$ with $\Sigma^{d-1}_N$. The analysis simplifies in this symmetric case. We fix a player and call it the \emph{tagged player}, and we refer to the untagged players as \emph{other players}. The state process of the problem can be described by a Markov chain $(X_t,Z_t)$ taking values in $\sX_N$. $X_t$ now denotes the state of the tagged player at time $t$ while $Z_t$ refers to the distribution of other players across states at time $t$. Given a feedback control $\alpha(t,x,z)$ of the tagged player and a feedback control $\beta(t,x,z)$ that is used by every other player, the transition probabilities are given by
$$
\P^{\alpha\otimes\beta}[X_{t+h} = x',Z_{t+h}=z' \, |\, X_t =x, Z_t=z] = \alpha_{x'}(t,x,z)\,\gamma^\beta_{z'}(t,x,z),
$$
where $\gamma^\beta(t,x,z)\in\Sigma^{d-1}$ denotes the probability distribution of a random variable $Z'(x,z,\beta(t))$, constructed as follows. Here, with an abuse of notation, we let $\alpha\otimes\beta$ denote the joint control in which the tagged player uses $\alpha$ and all other players use $\beta$.  Fix a time $t\in\cT$, a position of the tagged player $x\in\cX$ and an empirical distribution $z\in\Sigma^{d-1}_N$ of untagged players. At time $t$, there are $Nz_y$ untagged players in a state $y\in\cX$. Given that the tagged player is in state $x$, each untagged player in state $y$ chooses a next state independently according to the probability law $\beta(t,y,z+\frac{1}{N}(e_x-e_y))$, where $e_x\in\{0,1\}^d$ denotes the $x$th unit vector. It is clear that the distribution of these agents is given by the random variable $Z'(x,y,z,\beta(t))$ with 
$$
NZ'(x,y,z,\beta(t)) \sim \mathrm{Multinomial} (Nz_y,\beta(t,y,z+e^N_{xy})),\quad \text{where} \quad  e_{xy}^N:=\frac1N(e_x-e_y).
$$
We obtain the distribution of untagged players at time $t+h$ by summing over all possible states $y$:
\be \label{eq:multinomial}
Z'(x,z,\beta(t))=\sum_{y\in\cX}Z'(x,y,z,\beta(t)),
\ee
recalling that $\{Z'(x,y,z,\beta(t))\}_{y\in\cX}$ are independent. 

In the following, we will allow for restrictions on the class of admissible transition probabilities. 

\begin{Def}[and assumption]\label{def:adm-controls-finite-player}
{\rm
For each $x\in\cX$ and $h>0$, let $A(h,x)$ be a closed, convex and non-empty subset of $\Sigma^{d-1}$. 
\begin{itemize}
\item The set of all \emph{admissible feedback rules} $\sA_N=\sA_N(h)$ is the set of all maps $\alpha: \cT\times\sX_N \to \Sigma^{d-1}$ such that $\alpha(t,x,z)\in  A(h,x) $ for all $(t,x,z)\in\cT\times\sX_N$. 
\item For each $x\in\cX$, let $S(x)\subset\cX$ denote the support of $A(h,x)$, i.e., $S(x)$ consists of all $y\in\cX$ such that there exists a transition probability $a\in A(h,x)$ with $a(y)>0$. We assume that $S(x)$ is independent of $h>0$.
\item We set $m:=\max_{x\in\cX} |S(x)|$.
\end{itemize}
}
\end{Def}

Note that setting accommodates models with a \emph{minimum transition probability $h\sigma^2>0$} by setting $ A(h,x) \equiv A(h)$ with $A(h)= \{a\in\Sigma^{d-1}\,:\,a(y)\geq h\sigma^2 \ \forall y\in\cX\}$.

\begin{Ex}\label{ex:nearest-neighbor-I}
{\rm
For a nearest-neighbor model on the one-dimensional torus $S^1\cong[0,2\pi)$ we can let
$$
\cX= \Bigl\{0,\frac{2\pi}{d}, \frac{4\pi}{d},\ldots,\frac{(d-1)2\pi}{d} \Bigr\},\quad A(h,x) \subset \{a\in\Sigma^{d-1} \, : \, a_y=0\ \ \text{if}\ \ |y-x|>2\pi/d\}.
$$
Here, addition is defined modulo $2\pi/d$ and we have 
$$
S(x) = \{x-2\pi/d,x,x+2\pi/d\},\quad m=3.
$$ 
}
\end{Ex}

\begin{Rmk}
{\rm In general, when considering a class of models varying with $d$, note that the upper bound $m=m(d)$ on the allowed transitions may depend on $d$. In contrast, in the above Example \ref{ex:nearest-neighbor-I}, the number $m$ is independent of the dimension $d$.
}
\end{Rmk}

To unify the treatment with the mean-field case, we let running and terminal cost functions be given:
$$
\ell: \sX \times [0,\infty)^{d}\to [0,\infty),\quad g:\sX\to [0,\infty),\quad \text{where}\quad \sX:=\cX\times\Sigma^{d-1}.
$$

Now, given a control $\beta=\beta(t,x,z)$, the tagged player minimizes its \emph{cost functional}
\be\label{eq:cost-functional-tagged-player}
J^N(t,x,z,\alpha;\beta)  := \E^{\alpha\otimes\beta} \Bigl[\sum_{t\leq s<T} \ell(X_s,Z_s,\alpha(s,X_s,Z_s)/h)h + g(X_T,Z_T)\,|\,(X_t,Z_t)=(x,z)\Bigr]
\ee
over $\alpha\in\sA_N$. We define the \emph{value function} of the tagged player by
\be\label{eq:value-function-tagged-player}
v^\beta(t,x,z):=\inf_{\alpha\in\sA_N}\ J^N(t,x,z,\alpha;\beta),\quad (t,x,z)\in\bar\cT\times\sX_N.
\ee

\begin{Def}[Symmetric Markov perfect equilibrium]
A \emph{symmetric Markov perfect (Nash) equilibrium} for the $(N+1)$-player game is characterized by a single feedback rule $\alpha^*\in\sA_N$ such that $J^N(t,x,z,\alpha^*;\alpha^*)=\inf_{\alpha\in\sA_N}\, J^N(t,x,z,\alpha;\alpha^*)$ for all $(t,x,z)\in\cT\times\sX_N$.
\end{Def}

In the same way that the Nash system characterizes Markov perfect equilibria, there is a single equation, which we call the \emph{Nash-Lasry-Lions equation}, whose solution corresponds to symmetric Markov perfect equilibria. It reads, for $(t,x,z)\in \cT\times\sX_N$,
\begin{equation}\tag{$N$-NLL}\label{eq:N-NLL}
\left\{
\begin{split}
v(t,x,z)&= \underset{a\in  A(h,x) }{\inf} \, H(x,z,E(x,z,v(t+h),\alpha(t)),a),\\
\alpha(t,x,z) &= \underset{a\in  A(h,x) }{\mathrm{arg\,min}} \, H(x,z,E(x,z,v(t+h),\alpha(t)),a),\\
v(T,x,z)&=g(x,z),
\end{split} 
\right.
\end{equation}
where, for $(\beta,\varphi):\sX_N\to\Sigma^{d-1}\times\R$,
$$
E(x,z,\varphi,\beta) := \bigl(\E[\varphi(y,Z'(x,z,\beta))]\bigr)_{y\in\cX} \in \R^d,
$$
the expectation is taken with respect to $Z'(x,z,\beta)$ introduced in equation \eqref{eq:multinomial}, and for $(x,\mu,v,a)\in\sX\times\R^d\times\Sigma^{d-1}$, we defined another Hamiltonian by 
\be\label{eq:NLL-hamiltonian}
H(x,\mu,v,a) = \ell(x,\mu,a/h)h +  a\cdot v.
\ee
The unknowns of \eqref{eq:N-NLL} are $(v,\alpha)$, and we note that $\alpha(t)$ appears on both sides of the equation in the second line.

The following assumptions underlie most of our results in this paper. Recall the notation $\sX=\cX\times\Sigma^{d-1}$.

\begin{Asm}
\label{asm:basic}
\phantom{.}
\begin{enumerate}[(i)]
\item (Structural). \label{asm:basic-(i)} For every $(x,\mu,a)\in\sX\times\R^d$, the running cost $\ell(x,\mu,a)$ is independent of $a_x$.
\item (Continuity). The costs $\ell,g$ are continuous.
\item (Unique minimizer). \label{asm:basic-(iii)}
The function $ A(h,x) \ni a\mapsto H(x,\mu,v,a)$ admits a unique minimizer for fixed $h>0$ and $(x,\mu,v)\in\sX\times\R^d$.
\item (Boundedness).\label{asm:basic-(iv)}  The terminal cost $g$ is bounded. Furthermore, for every $(x,\mu,h)\in\sX\times(0,\infty)$, there exists an $a^*(x,\mu,h) \in A(h,x)$ such that 
$$
\limsup_{h\downarrow0}\sup_{(x,\mu)\in\sX} |\ell(x,\mu,a^*(x,\mu,h)/h)| h <\infty.
$$
\end{enumerate}
\end{Asm}

The structural assumption \ref{asm:basic} \eqref{asm:basic-(i)} is used to pass to the continuous-time limit. The boundedness assumption \ref{asm:basic} \eqref{asm:basic-(iv)} implies that, for any $\beta\in\sA_N$, the value function $v^\beta$ is uniformly bounded in $(t,x,\mu,h)\in\cT\times\sX\times(0,\bar h)$ for any finite $\bar h<\infty$. If $\ell(x,\mu,\cdot)$ is quadratic in $a_{-x}$ and transition probabilities of the form $\sigma^2h$, for some $\sigma^2\geq0$, are admissible, then $a^*_y(x,\mu,h)=\sigma^2h$ for $y\neq x$ satisfies the above assumption.

\begin{Rmk}[Correspondence]\label{rmk:correspondence-finite-player}
Under Assumption \ref{asm:basic}, standard results on dynamic programming (see Lemma \ref{lem:bellman}) imply that a given control $\alpha^*\in\sA_N$ is a symmetric Markov perfect equilibrium if and only if there exists a solution $v:\bar\cT\times\sX_N\to[0,\infty)$ to \eqref{eq:N-NLL} with $\alpha=\alpha^*$.
\end{Rmk}

For our uniqueness results, we need the following stronger version of Assumption \ref{asm:basic} \eqref{asm:basic-(iii)}.

\begin{Asm}[Strong convexity of the running cost]\label{asm:convex-cost}
There exists a $\gamma>0$ such that 
\be\label{eq:asm-convex-cost}
\ell(x,\mu, a') \geq \ell(x,\mu,a)+ \phi\cdot ( a'-a) + \gamma | a'-a|_1| a'-a|_\infty
\ee
for all $(x,\mu,a, a')\in\sX\times\R^{2d}$ and $\phi\in\partial_a\ell(x,\mu,a)$ that satisfy $a_y= a'_y=0$ for all $y\notin S(x)$ and $\sum_{y\in\cX} a_y= \sum_{y\in\cX} a'_y$ .  Here, $\partial_a\ell(x,\mu,a)$ denotes the subgradient of $\ell(x,\mu,\cdot)$ at the point $a$.
\end{Asm}

\begin{Rmk}
{\rm 
Recall $m$ from Definition \ref{def:adm-controls-finite-player}. Note that we have 
$$
A(h,x) \subset \{a\in\Sigma^{d-1} \, :\, |a|_0\leq m\},\quad \forall x\in\cX.
$$
If $m<d$, then this restricts the number of possible transitions at each state. Let $\ell:\sX\times \R^d\to[0,\infty)$ be a mapping that, instead of \eqref{eq:asm-convex-cost}, satisfies the following definition of strong convexity with respect to the $|\cdot|_2$-norm:
$$
\ell(x,\mu, a')\geq \ell(x,\mu,a) +y\cdot ( a'-a) + \gamma_0 |a'-a|_2^2.
$$
Then, $\ell$ satisfies Assumption \ref{asm:convex-cost} with $\gamma = \gamma_0/\sqrt{m}$}. This follows by noting that $|a|_1|a|_\infty\leq \sqrt{|a|_0}|a|_2^2$ holds for any $a\in\R^d$.
\end{Rmk}



\begin{Ex}\label{ex:nearest-neighbor-II}
In the setting of Example \ref{ex:nearest-neighbor-I}, define a quadratic cost function  $\ell(x,a) = \frac12 \sum_{y\in S(x) \setminus\{x\}}a_y^2.$
Then Assumption \ref{asm:convex-cost} holds with
\begin{align*}
\ell(x, a')-\ell(x,a) &= \frac12 \sum_{y\in S(x)\setminus\{x\}}((a'_y)^2-a_y^2) = \nabla_a\ell(x,a)\cdot (a '-a)+ \frac12 \sum_{y\in S(x)\setminus\{x\}} (a'_y-a_y)^2\\
&\geq\nabla_a\ell(x,a)\cdot (a'-a) + \frac{1}{6} |a'- a|_2^2\geq \nabla_a\ell(x,a)\cdot (a'-a) + \frac{1}{6\sqrt{3}}|a'-a|_1|a'-a|_\infty,    
\end{align*}
for all $(a,\tilde a)\in\R^{2d}$ that satisfy 
$a_y=\tilde a_y=0$ for any $y\notin S(x)$ and $\sum_y\tilde a_y=\sum_ya'_y$. Note that $\nabla_a\ell(x,a)$ is a $d$-dimensional vector whose $x$th coordinate is $0$.
\end{Ex}

\begin{Thm}\label{thm:symmetric-Nash-system-and-sMPE}
Under Assumption \ref{asm:basic}, the following hold.
\begin{enumerate}[(i)]
\item \label{thm:symmetric-Nash-system-and-sMPE-(i)} \emph{Existence:} Symmetric Markov perfect equilibria exist.
\item \label{thm:symmetric-Nash-system-and-sMPE-(ii)} \emph{Uniqueness:} In addition, let Assumption \ref{asm:convex-cost} hold. Then there exists an $h_*(N,T)>0$ such that, for every $h<h_*(N,T)$, the solution $v$ to the Nash-Lasry-Lions equation \eqref{eq:N-NLL} is unique. 
\end{enumerate}
\end{Thm}

\begin{Rmk}
\rm{
 The proof of Theorem \ref{thm:symmetric-Nash-system-and-sMPE} \eqref{thm:symmetric-Nash-system-and-sMPE-(ii)} reveals that one can equivalently express the smallness condition on $h$ in terms of the convexity $\gamma$ of the running cost in the control variable $a$ being sufficiently large.
}
\end{Rmk}

Theorem \ref{thm:symmetric-Nash-system-and-sMPE} implies the following corollary.

\begin{Cor}
Let Assumptions \ref{asm:basic} and \ref{asm:convex-cost} hold.
For $N\geq1$ and $h>0$, let $T^*(N,h)\in h\Z_+\cup\{\infty\}$ denote the maximal time duration such that \eqref{eq:N-NLL} has a unique solution on $\cT=\{0,h,\ldots,T^*(N,h)\}$. Then $T^*(N,h)\to\infty$ as $h\downarrow0$.
\end{Cor}

\subsection{The mean-field problem}

In the mean-field problem, a \emph{typical player} faces a continuum of agents that generate a flow of probability distributions. The assumption is that no player by itself influences the population distribution. At time $t\in\cT$ and state $x\in\cX$, a typical agent takes a flow of probability distributions $\boldsymbol\mu=(\mu_t,\ldots,\mu_T)$ as given and solves
\be\label{eq:MF-value-function}
\begin{split}
v(t,x;\boldsymbol\mu)&:= \inf_{\alpha\in\sA} \ J^\infty(t,x,\alpha;\boldsymbol\mu),\quad \text{where} \\
J^\infty(t,x,\alpha;\boldsymbol\mu)&:= \E^\alpha \Bigl[\sum_{t\leq s <T} \ell(X_s,\mu_s,\alpha(s,X_s)/h)h + g(X_T,\mu_T) \, | \, X_t=x
\Bigr],
\end{split}
\ee
and $\sA$ is the set of all feedback rules $\alpha: \cT\times\cX \to \Sigma^{d-1}$ such that $\alpha(t,x)\in  A(h,x) $ for all $(t,x)\in\cT\times\cX$. In the following, we write 
$$
\cT(t,T) := \{t,t+h,\ldots,T\}.
$$

\begin{Def}
{\rm
A \emph{solution to the mean-field game} or \emph{mean-field game (Nash) equilibrium (MFG-NE)}, started from $\mu\in\Sigma^{d-1}$ at time $t\in\cT$, is a measure flow $\boldsymbol\mu^*=(\mu^*_t,\ldots,\mu^*_T)$ such that there exists a control $\alpha^*\in\sA$ satisfying the following.
\begin{enumerate}[(i)]
\item Optimality: $J^\infty(t,x,\alpha^*;\boldsymbol\mu^*)\leq J^\infty(t,x,\alpha;\boldsymbol\mu^*)$ for any $\alpha\in\sA$ and any $x\in\cX$;
\item Consistency: For all $t\leq s\leq T$, we have
$$
\mu^*_s(x) = \sum_{y\in\cX} \mu(y) \, \P^{\alpha^*}[X_s=x\,|\,X_{t}=y],\quad x\in\cX.
$$
\end{enumerate}
}
\end{Def}

\begin{Rmk}
{\rm
Instead of the measure flow $\boldsymbol\mu^*$, it is equivalent to define a solution to the mean-field game starting from $(t,\mu)$ in terms of a feedback control $\alpha^*\in\sA$ such that $J^\infty(t,x,\alpha^*;\boldsymbol\mu^*)\leq J^\infty(t,x,\alpha;\boldsymbol\mu^*)$ for any  $\alpha\in\sA$ and $x\in\cX$, where $\mu^*_t=\mu$ and $\mu^*_s=\P^{\alpha^*}\circ X_s^{-1}$ for all $t\leq s\leq T$.
}
\end{Rmk}

Consistent with the finite-player setting, we define the Hamiltonian $H$ by
$$
H(x,\mu,v,a) = \ell(x,\mu,a/h)h +  a\cdot v,\quad (x,\mu,v,a)\in\sX\times\R^d\times\Sigma^{d-1}.
$$
It is well known, see Remark \ref{rmk:MFG-system-correspondence} below, that Nash equilibria correspond to solutions of the forward-backward \emph{mean-field game system} consisting of dynamic programming and Kolmogorov equations for $x\in\cX$ and $t\leq s<T$:
\begin{equation}\tag{MFG system}\label{eq:MFG-system}
\left\{
\begin{split}
v(s,x)&= \underset{a\in  A(h,x) }{\inf} \, H(x,\mu_s,v(s+h),a),\\
\alpha(s,x) &= \underset{a\in  A(h,x) }{\mathrm{arg\,min}} \ H(x,\mu_s,v(s+h),a),\\
\mu_{s+h}(x) &= \sum_{y\in\cX}\mu_s(y)  \alpha_x(s,y),\\
\mu_t&=\mu,\ v(T,x)=g(x,\mu_T).
\end{split}
\right.
\end{equation}

\begin{Rmk}[Correspondence] \label{rmk:MFG-system-correspondence}
{\rm
More precisely, 
under Assumption \ref{asm:basic} and given initial data $(t,\mu)\in\cT\times\Sigma^{d-1}$, a measure flow $\boldsymbol{\mu}\in(\Sigma^{d-1})^{\cT(t,T)}$ is a MFG-NE starting from initial distribution $\mu$ at time $t$ if and only if there exists a function $v:\cT(t,T)\times\cX\to[0,\infty)$ such that $(v,\boldsymbol{\mu})$ solve \eqref{eq:MFG-system} with initial data $(t,\mu)$.
}
\end{Rmk}

Finally, we introduce a solution concept that is the mean-field analogue of the finite-player Markov perfect equilibria. To this end, we let $\sA_\infty$ be the set of all feedback rules $\bar\alpha: \cT\times\sX \to \Sigma^{d-1}$ such that $\bar \alpha(t,x,\mu)\in  A(h,x) $ for all $(t,x,\mu)$.  Fix an initial condition $(t,x)\in\cT\times\cX$ of a typical player and a distribution of the population $\boldsymbol{\mu}=(\mu_s)_{t\leq s\leq T}$. Given a control $\bar \alpha\in\sA_\infty$, define the control $\alpha^{\bar \alpha, \boldsymbol\mu}(s,x) := \bar \alpha(s,x,\mu_s)$
and set 
$$
J^\infty(t,x,\bar \alpha;\boldsymbol\mu) := J^\infty(t,x,\alpha^{\bar \alpha, \boldsymbol\mu};\boldsymbol\mu).
$$

\begin{Def}
A \emph{Markov perfect (Nash) equilibrium for the mean-field game (MF-MPE)} is a feedback policy $\bar \alpha \in\sA_\infty$ such that, for any $(t,x,\mu)\in\cT\times\sX$ and any control $\alpha\in\sA$, we have
$$
J^\infty(t,x,\bar \alpha;\boldsymbol\mu^{t,\mu,\bar \alpha})\leq J^\infty(t,x,\alpha;\boldsymbol\mu^{t,\mu,\bar \alpha}),
$$
with $\boldsymbol\mu^{t,\mu,\bar \alpha}$ being the distribution of $X$ under $\P^{\bar\alpha}$, i.e., for $t\leq s < T$,
$$
\mu_{s+h}^{t,\mu,\bar \alpha}(x) = \sum_{y\in\cX} \mu_{s}^{t,\mu,\bar \alpha}(y) \bar \alpha_x(s,y,\mu_{s}^{t,\mu,\bar \alpha}),\quad \mu^{t,\mu,\bar \alpha}_t=\mu.
$$
\end{Def}

If $\bar \alpha\in\sA_\infty$ is a Markov perfect equilibrium for the mean-field game, then the value function $V(t,x,\mu) = J^\infty(t,x,\bar \alpha; \boldsymbol\mu^{t,\mu,\bar \alpha})$
is characterized by the \emph{mean-field Nash-Lasry-Lions equation} or \emph{master equation}:
\begin{equation}\tag{MF-NLL}\label{eq:MF-NLL}
\left\{
\begin{aligned}
V(t,x,\mu) &= \underset{a\in  A(h,x) }{\inf} \,  H(x,\mu,V(t+h,\cdot,\mu\cdot \bar \alpha(t,\cdot,\mu)),a), \\
\bar \alpha(t,x,\mu) &= \underset{a\in  A(h,x) }{\mathrm{arg\,min}} \,H(x,\mu,V(t+h,\cdot,\mu \cdot\bar \alpha(t,\cdot,\mu)),a),\\
V(T,x,\mu)&=g(x,\mu).
\end{aligned}
\right.
\end{equation}
Here, $(V,\bar\alpha)$ are the unknowns. We interpret $\mu$ as a row vector and $\bar\alpha(t,\cdot,\mu)\in\R^{d\times d}$ as a transition matrix so that each $\bar\alpha(t,y,\mu)$ is a row vector. 
By including the distribution of the population as a state variable, the following remark follows by dynamic programming.

\begin{Rmk}[Correspondence]\label{rmk:correspondence-MF-NLL}
{\rm
Let Assumptions \ref{asm:basic} hold and let $ \bar\alpha \in\sA_\infty$. Then, $ \bar\alpha $ is a MF-MPE if and only if there exists a function $V:\bar\cT\times\sX\to [0,\infty)$ such that $(V, \bar\alpha )$ satisfies \eqref{eq:MF-NLL}.
}
\end{Rmk}

\begin{Prop}\label{prop:MFG-NE and MF-NLL}
Let Assumptions \ref{asm:basic} hold and fix an initial condition $(t,\mu)\in\cT\times\Sigma^{d-1}$.
\begin{enumerate}[(i)]
\item \label{prop:MFG-NE and MF-NLL-(i)} There exists a solution to the mean-field game starting from $\mu$ at time $t$. 
\item \label{prop:MFG-NE and MF-NLL-(ii)} Let $(V,\bar\alpha)$ be a solution to \eqref{eq:MF-NLL}. Then, the pair $(v^*,\mu^*)$, defined by
$$
\mu^*_{t} := \mu,\quad \mu^*_{s+h}(x) := \sum_{y\in\cX} \mu^*_s(y)\bar\alpha_x(s,y,\mu^*_s),\quad v^*(s,x):=V(s,x,\mu^*_s),\qquad s\leq t<T,
$$
is a solution to \eqref{eq:MFG-system}.
\end{enumerate}
\end{Prop}

The implications above can be illustrated as follows. 

\begin{figure}[h]
\centering
\begin{tikzpicture}[node distance=4cm, every node/.style={draw=none, minimum width=2cm, minimum height=1cm, align=center}]
\node (box1) {MF-MPE};
\node (box2) [right of=box1] {solutions to \\ MF-NLL};
\node (box3) [right of=box2] {solutions to \\ MFG system};
\node (box4) [right of=box3] {MFG-NE};

\draw[<-> , thick, shorten <=10pt, shorten >=10pt](box1) -- (box2);
\draw[-> , thick, shorten <=10pt, shorten >=10pt] (box2) -- (box3);
\draw[<-> , thick, shorten <=10pt, shorten >=10pt] (box3) -- (box4);
\end{tikzpicture}
\end{figure}

\subsection{Short-time results}

We proceed to provide conditions under which there exists a unique mean-field MPE, or equivalently, a unique solution to the mean-field NLL equation.

\begin{Asm}[Lipschitzness] \label{asm:lipschitz-cost}
There exist $0\leq L_\ell,L_g,L_{\partial\ell}<\infty$ such that, for every $(x,a,\mu,\tilde\mu)\in\cX\times [0,\infty)^d\times \Sigma^{d-1}\times\Sigma^{d-1}$ and $w\in\partial_a\ell(x,\mu,a)$, $\tilde w\in\partial_a\ell(x,\tilde \mu, a)$, we have
\begin{align*}
|\ell(x,\tilde \mu, a)-\ell(x,\mu,a)| \leq L_\ell |\tilde \mu-\mu|_1,  \ \sum_{y\in S(x)}  \!\! |\tilde w(y)-w(y)| \leq L_{\partial\ell} |\tilde \mu-\mu|_1, \ |g(x,\tilde \mu)-g(x, \mu)| \leq L_g |\tilde \mu-\mu|_1.
\end{align*}
As before, $\partial_a\ell(x,\mu,a)$ refers to the subgradient of $\ell(x,\mu,\cdot)$ at the point $a$.
\end{Asm}

Under these assumptions, we have the following existence and uniqueness result in short time, for both the finite-player and mean-field problems.

\begin{Thm}\label{thm:well-posedness-MF-NLL}
Let Assumptions \ref{asm:basic}, \ref{asm:convex-cost} and \ref{asm:lipschitz-cost} hold. Then, there exist constants $h_*,T_*>0$, depending on the data and $m$ but not on the number of players $N+1$, such that, whenever $h<h_*$ and $Kh<T_*$, there exists a unique solution to \eqref{eq:MF-NLL} and \eqref{eq:N-NLL} on $\bar\cT=\{0,h,\ldots,Kh\}$, for any $N\geq1$.
\end{Thm}

\subsection{Mean-field convergence} \label{ssec:mean-field-convergence}

Fix the time horizon $T$ and step size $h$. For each $N$, any solution $(v^{(N)},\alpha^{(N)})$  to \eqref{eq:N-NLL} is extended from $\bar\cT\times\sX_N$ to $\bar \cT\times\sX$ by linear interpolation in the $z$-variable. 

\begin{Thm}\label{thm:N-convergence}
Let the assumptions of Theorem \ref{thm:well-posedness-MF-NLL} hold, and let $h_*,T_*$ be as in that theorem. Let $h<h_*$ and $Kh<T_*$ be given. Let $(V,\alpha)$ and $(v^{(N)},\alpha^{(N)})$ be the solutions to \eqref{eq:MF-NLL} and \eqref{eq:N-NLL}, respectively. Set $\bar\cT=\{0,h,\ldots,Kh\}$. Then the following hold:
\begin{enumerate}[(i)]
\item (Convergence of values) As $N\to\infty$, $ v^{(N)}\to V$ on $\bar\cT\times\sX$ uniformly.
\item (Convergence of MPE) As $N\to\infty$, $ \alpha^{(N)}\to \alpha$ on $\cT\times\sX$ uniformly.
\end{enumerate}
\end{Thm}

\subsection{Continuous-time limit} \label{ssec:continuous-limit}

We now fix the time horizon $T$ and the number of players $N+1$, and we consider the continuous-time limit $h\downarrow 0$. We continue to briefly define the limiting problem and refer to \cite{GMS1} for a detailed discussion. In the continuous-time model, players' state processes follow a time-inhomogeneous continuous-time Markov chain, which, by an abuse of notation, we again denote by $(X_t,Z_t)$. For a given \emph{thermal noise} level $\sigma^2\geq0$, let $\sA_{\mathrm{cts}}$ be the set of all \emph{admissible feedback controls} $\beta$, that is, Lipschitz continuous functions $\beta:[0,T]\times\sX_N\to\R^d$ satisfying $\beta(t,x,z)\in\cR(\sigma^2,x)$ for any $(t,x,z)\in[0,T]\times\sX_N$. Here, the set of \emph{admissible transition rates} from a state $x$ is defined by
$$
\cR(\sigma^2,x):=\Big\{a\in\R^d: a_y\geq\sigma^2 \ \text{for all}\ y\in S(x)\setminus\{x\}, \ a_y=0 \ \text{for all} \ y\in\cX\setminus (S(x) \cup\{x\}) \ \text{and}\  a_x = -\sum_{y\neq x}a_y\Big\},
$$
where, as before, $S(x)\subset\cX$ denotes a (fixed) set of allowable transitions. In words, $\beta_{y}(t,x,z)$ is the rate of changing from the current state $x\in\cX$ to a state $y\neq x$ at time $t\in[0,T]$, given that the empirical distribution of other players is $z\in\Sigma^{d-1}_N$. Given a control $\alpha\in\sA_{\mathrm{cts}}$ of the tagged player and a control $\beta\in\sA_{\mathrm{cts}}$ used by all other players, for any $(t,x,z)\in[0,T]\times\sX_N$ and $y,w,x'\in\cX$, we have
$$
\P^{\alpha\otimes\beta}[X_{t+\Delta t} = x' \, |\, (X_t,Z_t)=(x,z)] = \alpha_{x'}(t,x,z) \, \Delta t + o(\Delta t),\quad \text{as} \ \Delta t \downarrow0,
$$
and
$$
\P^{\alpha\otimes\beta}[Z_{t+\Delta t} = z+ e^N_{wy} \, |\, (X_t,Z_t)=(x,z)] = Nz_y \beta_w(t,y,z+e^N_{xy}) \, \Delta t + o(\Delta t),\quad \text{as} \ \Delta t \downarrow0,
$$
where we recall the definition $e^N_{xy}=\frac1N(e_x-e_y)$, and $o(\Delta t)/\Delta t\to0$ as $\Delta t\downarrow0$. The tagged player solves
\begin{equation*}
\begin{split}
v^\beta_{\mathrm{cts}}(t,x,z)&:= \inf_{\alpha\in\sA_{\mathrm{cts}}}  \ \ J_{\mathrm{cts}}(t,x,z,\alpha;\beta),\quad \text{where} \\
J_{\mathrm{cts}}(t,x,z,\alpha;\beta) &:= \E^{\alpha\otimes\beta} 
\Bigl[
\int_t^T\ell(X_s,Z_s,\alpha(s,X_s,Z_s))\,\d s + g(X_T,Z_T) \, | \, (X_t,Z_t)=(x,z)
\Bigr].
\end{split}
\end{equation*}
In full analogy with the discrete-time definition, a control $\alpha^*\in\sA_{\mathrm{cts}}$ is called a \emph{Markov perfect (Nash) equilibrium (MPE)} if
$$
J_{\mathrm{cts}}(t,x,z,\alpha^*;\alpha^*) \leq J_{\mathrm{cts}}(t,x,z,\alpha;\alpha^*)\quad \forall \alpha\in\sA_{\mathrm{cts}} \ \forall (t,x,z)\in [0,T]\times \sX_N.
$$
Given $\beta:\sX_N\to\R^d$, we define the operator $\sL^\beta$ by setting
\begin{equation}\label{eq:generator}
(\sL^\beta\varphi)(x,z) := \sum_{ w,y\in\cX:\, w\neq y} Nz_y \beta_w(y,z+e^N_{xy}) [\varphi(x,z+e^N_{wy})-\varphi(x,z)],\quad (x,z)\in\sX_N,
\end{equation}
for any $\varphi:\sX_N\to\R$, 
and we define the Hamiltonian $\cH$ by
$$
\cH(x,z,v,a) :=  \ell(x,z,a) + a\cdot v,\quad (x,z,v,a)\in\sX_N\times\R^d\times\R^d.
$$
Finally, we denote the minimizer of the Hamiltonian by
$$
\hat \alpha:\sX_N\times\R^d\to\R^d,\qquad \hat \alpha(x,z,v):= \underset{a\in\cR(\sigma^2,x)}{\mathrm{arg\,min}} \ \cH(x,z,v,a).
$$
Gomes, Mohr and Souza \cite{GMS1} show that there is a unique MPE which corresponds to the unique classical solution $v_{\mathrm{cts}}:[0,T]\times\sX_N\to[0,\infty)$ of the following continuous-time finite-player NLL equation:
\be\label{eq:cts-time-finite-player-NLL}
\left\{
\begin{split}
-\frac{\d}{\d t} v_{\mathrm{cts}}(t,x,z) &= \underset{a\in\cR(\sigma^2,x)}{\inf}\cH(x,z,\Delta_x v_{\mathrm{cts}}(t,\cdot,z),a)+ \sL^{\alpha_{\mathrm{cts}}(t)} v(t) (x,z),\\
\alpha_{\mathrm{cts}}(t,x,z) &= \hat\alpha(x,z,\Delta_xv_{\mathrm{cts}}(t,\cdot,z)),\\
v_{\mathrm{cts}}(T,x,z)&=g(x,z).
\end{split}
\right.
\ee
Here, we used the notation
$$
\Delta_x\varphi := (\varphi(y)-\varphi(x))_{y\in\cX}\in\R^d
$$
for any $\varphi:\cX\to\R$. Note that $\Delta_x\varphi(x)=0$.
We continue to state our convergence result. For $K\geq1$, define $h=h^{(K)}:=T/K$ and $\bar\cT^{(K)}:=\{0,h,\ldots,(K-1)h,T\}$ and we assume
\be\label{eq:cts-limit-adm-prob}
A(h,x)= \big\{a\in\Sigma^{d-1}:a_y\geq h\sigma^2\ \text{for all} \ y\in S(x)\setminus\{x\} \ \text{and} \ a_y=0 \ \text{for all} \ y\in \cX\setminus(S(x)\cup\{x\}) \big\}.
\ee
For sufficiently large $K$, let $(v^{(K)},\alpha^{(K)})$ be the unique solution to \eqref{eq:N-NLL} on $\bar\cT^{(K)}$, guaranteed by Theorem \ref{thm:symmetric-Nash-system-and-sMPE}, with the aforementioned choice of $A(h,x)$, and extend $(v^{(K)},\alpha^{(K)})$ to $[0,T]\times\sX_N$ by linear interpolation in the time variable. Set $$
\mathfrak{a}^{(K)}(t,x,z):=\alpha^{(K)}(t,x,z)-e_x,\quad (t,x,z)\in[0,T]\times\sX_N,\ K\geq1.
$$

\begin{Thm}\label{thm:continuous-convergence}
Under Assumptions \ref{asm:basic}, \ref{asm:convex-cost} and \ref{asm:lipschitz-cost} we have the following.
\begin{enumerate}[(i)]
\item (Convergence of values) As $K\to\infty$, $v^{(K)}\to v_{\mathrm{cts}}$ uniformly on $[0,T]\times\sX_N$.
\item (Convergence of MPE) As $K\to\infty$, $\mathfrak{a}^{(K)}/h^{(K)}\to \alpha_{\mathrm{cts}}$ uniformly on $[0,T]\times\sX_N$.
\end{enumerate}
\end{Thm}

\subsection{Comparison with existing works} \label{ssec:literature}

We study stochastic finite-player and mean-field games in discrete time and space, where players control transition probabilities over a finite time horizon. Similar settings have been explored in the literature. For instance, Saldi, Ba\c{s}ar and Raginsky \cite{SBR} examine a mean-field game on a general state space with a discounted infinite-horizon cost functional. Doncel, Gast and Gaujal \cite{doncel2019discrete} consider both discrete and continuous-time problems, but with finite action spaces. Our setup is closest to Gomes, Mohr and Souza \cite{GMS2} who study the exponential convergence of a discrete-time mean-field game to its stationary solution.

In Remarks \ref{rmk:correspondence-finite-player} and \ref{rmk:correspondence-MF-NLL}, we first clarify the connection between symmetric Markov perfect equilibria and the Nash-Lasry-Lions equation, both in the finite-player and the mean-field formulation. In the finite-player case, this relationship has been observed, see for example Gomes, Mohr and Souza \cite{GMS1} for a reference close to our setting. 

Next, we address the existence and uniqueness of MPE in our finite-player problem, see Theorem \ref{thm:symmetric-Nash-system-and-sMPE}. Existence follows from well-known fixed-point arguments, see for example Gomes, Mohr and Souza \cite{GMS2}, while our uniqueness proof relies on a contraction argument and uses the smallness of the time step. This result sheds light on the relationship with the continuous-time problem, which in the finite-player case becomes uniquely solvable, see \cite{GMS1}. In the mean-field, we obtain short-time results in Theorem \ref{thm:well-posedness-MF-NLL}, which can be made global under monotonicity assumptions. In this regard, our work is related to the pasting of equilibria by Dianetti, Nendel, Tangpi and Wang \cite{dianetti2024pasting}, who utilize a discrete-time master equation approach under Lasry-Lions monotonicity, see Remark \ref{rmk:monotonicity} below.

Finally, we present convergence results of the finite-player game in Theorems \ref{thm:N-convergence} and \ref{thm:continuous-convergence}, both as the number of players $N+1$ goes to infinity and as the time discretization $h$ goes to zero.  The former parallels the mean-field convergence obtained in \cite{GMS1} in continuous-time while in the latter limit we recover the uniquely solvable finite-player NLL equation in \cite{GMS1}.

\vspace{1em}
\textbf{Structure.} The rest of the paper is organized as follows. In Section \ref{sec:sym-finite-player} we investigate the symmetric finite-player problem and establish Theorem \ref{thm:symmetric-Nash-system-and-sMPE}. Section \ref{sec:mean-field-problem} discusses the mean-field problem and proves Proposition \ref{prop:MFG-NE and MF-NLL}. The short-time results recorded in Theorem \ref{thm:well-posedness-MF-NLL} for both finite-player and mean-field games are obtained in Section \ref{sec:short-time}. The mean-field limit $N\to\infty$, see Theorem \ref{thm:N-convergence}, is established in Section \ref{sec:mean-field-convergence}, and the continuous-time limit $h\downarrow0$, see Theorem \ref{thm:continuous-convergence}, is treated in Section \ref{sec:continuous-limit}. We recall some basic facts about dynamic programming in discrete time in Appendix \ref{app:dynamic-progamming} and prove a generalized conjugate correspondence theorem in Appendix \ref{app:conjugate-correspondence}.

\section{The symmetric finite-player problem}\label{sec:sym-finite-player}
This section proves Theorem \ref{thm:symmetric-Nash-system-and-sMPE}. While the proof of \eqref{thm:symmetric-Nash-system-and-sMPE-(i)} relies on results in dynamic programming and an application of Brouwer's fixed-point theorem, the analysis for part \eqref{thm:symmetric-Nash-system-and-sMPE-(ii)} makes crucial use of a Lipschitz property of the multinomial distribution presented in Lemmas \ref{lem:lipschitz-multinom-I}, \ref{lem:lipschitz-multinom-II}. 

\subsection{Existence}

We continue to establish the existence of symmetric MPE as claimed in Theorem \ref{thm:symmetric-Nash-system-and-sMPE} \eqref{thm:symmetric-Nash-system-and-sMPE-(i)} using a classical fixed-point argument.

\begin{proof}[Proof of Theorem \ref{thm:symmetric-Nash-system-and-sMPE} \eqref{thm:symmetric-Nash-system-and-sMPE-(i)}]
We define a map $\Phi:\sA_N\to\sA_N$ by the following procedure. For a given control $\beta\in\sA_N$ of the other players, let $\alpha=\alpha(\beta)\in\sA_N$ be the optimal unique feedback control of the tagged player. By Lemma \ref{lem:bellman}, $\alpha$ can be computed using the dynamic programming equation
$$
\begin{cases}
v^\beta(t,x,z)= \underset{a\in  A(h,x) }{\inf} \, H(x,z,E(x,z, v^\beta(t+h),\beta(t)),a),\\
\alpha(t,x,z) = \underset{a\in  A(h,x) }{\mathrm{arg\,min}} \ H(x,z,E(x,z, v^\beta(t+h),\beta(t)),a),\\
v^\beta(T,x,z)=g(x,z),
\end{cases} 
$$
by backward recursion. Then set $\Phi(\beta):= \alpha$. Using Berge's maximum theorem, it is easy to see that $\Phi$ is a continuous mapping. Since $\sA_N$ is a convex compact set, Brouwer's fixed-point theorem guarantees the existence of a fixed point. By Remark \ref{rmk:correspondence-finite-player}, any fixed point corresponds to a symmetric Markov perfect equilibrium.
\end{proof}

\subsection{Uniqueness}
By Remark \ref{rmk:correspondence-finite-player}, uniqueness of symmetric MPE is equivalent to uniqueness of solutions to \eqref{eq:N-NLL}. By (backward) induction, it is equivalent to uniqueness in one-step games. More precisely, it is sufficient to show that, for any $\varphi:\sX_N\to\R$, there exists a unique map $\hat \alpha:\sX_N\to\Sigma^{d-1}$ satisfying the fixed-point equation
$$
\hat \alpha(x,z) = \underset{a\in  A(h,x) }{\mathrm{arg\,min}} \ H(x,z,E(x,z,\varphi,\hat \alpha),a),\quad (x,z)\in\sX_N.
$$
We prove this using Banach's fixed-point theorem in Proposition \ref{prop:finite-player-one-step} below. To this end, recall from \eqref{eq:multinomial} that $Z'(x,z,\alpha)$ is defined as the sum of $d$ independent random variables $Z'(x,1,z,\alpha),\ldots,Z'(x,d,z,\alpha)$ with
$$
NZ'(x,y,z,\alpha) \sim \mathrm{Multinomial} (Nz_y,\alpha(y,z+e^N_{xy})),\quad y\in\cX.
$$

We now provide two Lipschitz estimates involving $Z'(x,z,\alpha)$.

\begin{Lem}\label{lem:lipschitz-multinom-I}
For any maps $\alpha,\tilde\alpha:\sX_N\to\Sigma^{d-1}$, we have
$$
\max_{(x,z)\in\sX_N} \E|Z'(x,z,\alpha)-Z'(x,z,\tilde \alpha)|_1 \leq  \max_{(x,z)\in\sX_N} |\alpha(x,z)-\tilde\alpha(x,z)|_1.
$$
\end{Lem}

\begin{proof}
First note that if $\eta\sim\mathrm{Multinomial}(n,p)$ for some $n\geq1$, $p\in\Sigma^{d-1}$, then for every coordinate $j\in\cX$,
$$
\eta_j \overset{d}{=} \sum_{i=1}^n  \chi_{[0,p_j]}(\xi^i) \in \Z_+,
$$
where $\{\xi^i\}$ are independent $\mathrm{Uniform}(0,1)$ random variables. Now fix $j\in\cX$ and $(x,z)\in\sX_N$, let $\xi^i_y$, $1\leq i\leq N,\, y\in\cX$, be a collection of independent $\mathrm{Uniform}(0,1)$ random variables, and write
\begin{align*}
Z'_j(x,z,\alpha)-Z'_j(x,z,\tilde \alpha) &= \sum_{y\in\cX} (Z'_j(x,y,z,\alpha)-Z'_j(x,y,z,\tilde \alpha))\\
&\overset{d}{=} \frac1N \sum_{y\in\cX} \sum_{i=1}^{Nz_y}(\chi_{[0,\alpha_j(y,z+e^N_{xy})]} -  \chi_{[0,\tilde \alpha_j(y,z+e^N_{xy})]} 
)(\xi^i_y).
\end{align*}
Taking expectations,
\begin{align*}
\E|Z'(x,z,\alpha) -Z'(x,z,\tilde\alpha)|_1 &=\sum_{j\in\cX} \E|Z'_j(x,z,\alpha) -Z'_j(x,z,\tilde \alpha)|\\
&\leq\frac{1}{N} \sum_{j\in\cX}\sum_{y\in\cX}
 \sum_{i=1}^{Nz_y} \, \E|(\chi_{[0,\alpha_j(y,z+e^N_{xy})]} -  \chi_{[0,\tilde \alpha_j(y,z+e^N_{xy})]} 
)(\xi^i_y)|\\
&= \sum_{y\in\cX}  z_y|\alpha(y,z+e^N_{xy}) - \tilde \alpha(y,z+e^N_{xy})|_1\\
&\leq \Bigl(\sum_{y\in\cX} z_y \Bigr) \max_{(x,z)\in\sX_N} \, |\alpha(x,z)-\tilde\alpha(x,z)|_1 \\
&= \max_{(x,z)\in\sX_N} \, |\alpha(x,z)-\tilde\alpha(x,z)|_1.
\end{align*}
Taking the maximum over $(x,z)$ finishes the proof.
\end{proof}

\begin{Lem}\label{lem:lipschitz-multinom-II}
For any map $\alpha:\sX_N\to\Sigma^{d-1}$ and any $z,\tilde z\in\Sigma^{d-1}_N$, we have
$$
\max_{x\in\cX} \, \E|Z'(x, z,\alpha)-Z'(x,\tilde z,\alpha)|_1\leq  (1+\max_{x\in\cX} \, \Lip_1(\alpha(x,\cdot)) )\,|z-\tilde z|_1.
$$
\end{Lem}

\begin{proof}
Using the same notation as in the proof of the previous lemma, for any $x\in\cX$,
\begin{align*}
\E|Z'(x, z,\,\alpha)-Z'(x,\tilde z,\alpha)|_1&= \sum_{j\in\cX} \E|Z'_j(x,z,\alpha)-Z'_j(x,\tilde z, \alpha)|\\
&\leq\frac1N \sum_{j\in\cX} \sum_{y\in\cX}\sum_{i=1}^{N(z_y\land\tilde z_y)} \E|(\chi_{[0,\alpha_j(y,z+ e^N_{xy} )]} -  \chi_{[0, \alpha_j(y,\tilde z+ e^N_{xy} )]})(\xi^i_y)|\\
& + \sum_{j\in\cX} \sum_{y\in\cX}\Bigl((z_y-\tilde z_y)^+ \,\alpha_j(y,z+ e^N_{xy} )+ ( z_y-\tilde z_y)^-\,\alpha_j(y,\tilde z+ e^N_{xy} )\Bigr)\\
&= \sum_{y\in\cX} (z_y\land \tilde z_y)  |\alpha(y,z+ e^N_{xy} )-\alpha(y,\tilde z+ e^N_{xy} )|_1+ |z-\tilde z|_1\\
&\leq \Bigl(1+ \sum_{y\in\cX} (z_y\land \tilde z_y) \Lip_1(\alpha(y,\cdot)) \Bigr)  |z-\tilde z|_1\\
&\leq  \, (1+\max_{y\in\cX} \Lip_1(\alpha(y,\cdot)) )\,|z-\tilde z|_1.
\end{align*}
Taking the maximum over $(x,z)$ again finishes the proof.
\end{proof}

We note the following lemma, whose proof is given in Appendix \ref{app:conjugate-correspondence}.

\begin{Lem}[Variant of the conjugate correspondence theorem] \label{lem:convex-conjugates}
Let $f:\R^d\to (-\infty,\infty]$ be a proper, closed and convex function. Let $C\subset\dom(f)$ be a non-empty, closed and convex set. Fix $p,q\in[1,\infty]$ with $1/p+1/q=1$ and assume that there exists a $\gamma>0$ such that  
\be\label{eq:lem-convex-conjugate}
f(x')\geq f(x)+\phi \cdot (x'-x) + \frac{\gamma}{2} |x'-x|_p|x'-x|_q
\ee
for any $x,x'\in C$ and $\phi\in\partial f(x)$. Set
$
\psi(y) = \mathrm{arg\,max} \{y\cdot x-f(x):x\in C\}$, $y\in\R^d$.
Then, $\psi$ is single-valued and is Lipschitz continuous with constant $1/\gamma$, with respect to both $|\cdot|_p$ and $|\cdot|_q$.
\end{Lem}

We are now ready to prove existence and uniqueness for one-step games.  We recall the constant $\gamma$ from Assumption \ref{asm:convex-cost}, $S(x)$ and $m$ from Definition \ref{def:adm-controls-finite-player}, and the constant $L_{\partial\ell}$ from Assumption \ref{asm:lipschitz-cost}.

\begin{Prop}\label{prop:finite-player-one-step}
Let Assumptions \ref{asm:basic}  and \ref{asm:convex-cost} hold. Then, for any $\varphi:\sX_N\to\R$, there exists a unique function $\hat\alpha:\sX_N\to\Sigma^{d-1}$ that satisfies
$$
\hat \alpha(x,z) = \underset{a\in  A(h,x) }{\mathrm{arg\,min}} \, H(x,z,E(x,z,\varphi,\hat \alpha),a),\quad \forall (x,z)\in\sX_N,
$$
provided that $h<\gamma/L_\varphi$, where $L_\varphi:= m \max_{x\in\cX}\Lip_1(\varphi(x,\cdot))$.  In addition, let Assumption \ref{asm:lipschitz-cost} hold. Then,
$$
\max_{x\in\cX} \, \Lip_1(\hat\alpha(x,\cdot)) \leq \frac{h(L_{\partial\ell}+ L_\varphi)}{\gamma-h L_\varphi}.
$$
\end{Prop}

\begin{proof}
Define
$$
\cS:=\{\alpha:\sX_N\to \Sigma^{d-1}\,:\,\alpha_y(x,z)=0\ \forall y\notin S(x), \, \forall (x,z)\in\sX_N\},
$$
For $\alpha,\tilde\alpha\in\cS$ set $d_\cS(\alpha,\tilde\alpha):=\max_{(x,z)\in\sX_N}|\alpha(x,z)-\tilde\alpha(x,z)|_1$,
$$
\Phi_{x,z}(\alpha):=\underset{a\in  A(h,x) }{\mathrm{arg\,min}} \, H(x,z,E(x,z,\varphi,\alpha),a) \quad \text{and}\quad \Phi(\alpha):= (\Phi_{x,z}(\alpha))_{(x,z)\in\sX_N}.
$$
Now let $(x,z)\in\sX_N$ be given. Define a function $f:\R^d\to[0,\infty]$ by $f(a):=\ell(x,z,a/h)h$ for $a\in[0,\infty)^d$ and $+\infty$ otherwise. Set $C:=A(h,x)$. By Assumptions \ref{asm:basic} and \ref{asm:convex-cost}, $f$ is a proper, closed and convex function. Using Assumption \ref{asm:convex-cost} and the fact that any element $a\in C$ satisfies $\sum_{y\in\cX}a_y=1$ and $a_y=0$ for all $y\notin S(x)$, we see that  
\be\label{eq:finite-player-one-step-(i)}
f(a')\geq f(a) + \phi\cdot (a'-a) + \frac{\gamma}{h} | a'-a|_1| a'-a|_\infty,\quad \forall (a, a')\in C\times C,\, \forall \phi\in\partial f(a).
\ee
Hence $f$ satisfies the conditions of Lemma \ref{lem:convex-conjugates}. Let $\alpha,\tilde\alpha\in \cS$ be given.
Observe that
$$
\Phi_{x,z}( \alpha)  =  \underset{a\in C }{\mathrm{arg\,max}} \, \Bigl\{\sum_{y\in S(x)}\! a_y \, (-\E[\varphi(y,Z'(x,z,\alpha))]) -f(a) \Bigr\}.
$$
By \eqref{eq:finite-player-one-step-(i)} and Lemma \ref{lem:convex-conjugates},
$$
|\Phi_{x,z}( \alpha)-\Phi_{x,z}(\tilde \alpha)|_1\leq \frac{h}{2\gamma}\sum_{y\in S(x)}\E|\varphi(y,Z'(x,z,\alpha))- \varphi(y,Z'(x,z,\tilde\alpha))|.
$$
Using Lemma \ref{lem:lipschitz-multinom-I},
\begin{align*}
\sum_{y\in S(x)} \E|\varphi(y,Z'(x,z,\alpha)) - \varphi(y,Z'(x,z,\tilde\alpha))| 
&\leq  \sum_{y\in S(x)} \Lip_1(\varphi(y,\cdot))\, \E|Z'(x,z,\alpha) - Z'(x,z,\tilde\alpha)|_1\\
&\leq m \, \max_{y\in\cX} \Lip_1(\varphi(y,\cdot)) \ \E|Z'(x,z,\alpha) - Z'(x,z,\tilde\alpha)|_1\\
&\leq  L_\varphi\, \max_{(x,z)\in\sX_N} |\alpha(x,z)-\tilde\alpha(x,z)|_1,
\end{align*}
with $L_\varphi= m \max_{y\in\cX}\Lip_1(\varphi(y,\cdot))$. We conclude that
$$
d_\cS(\Phi(\alpha),\Phi(\tilde\alpha)) \leq \frac{h}{2\gamma} \max_{(x,z)\in\sX_N} \sum_{y\in S(x)}\E|\varphi(y,Z'(x,z,\alpha))-\varphi(y,Z'(x,z,\tilde\alpha))|\leq \frac{L_\varphi h}{2\gamma}  d_\cS(\alpha,\tilde\alpha).
$$
Hence, $\Phi$ is a contraction since $h<\gamma/L_\varphi$, in which case Banach's fixed-point theorem guarantees  the uniqueness of a fixed point $\hat \alpha\in\cS$.

We proceed to establish an upper bound of the Lipschitz constant for $\hat\alpha(x,\cdot)$. Fix $z,\tilde z\in\Sigma^{d-1}_N$. By Assumption \ref{asm:convex-cost}, for any $a,\tilde a\in A(h,x)$ and $\tilde w\in\partial_a\ell(x,\tilde z,a/h)\subset\R^d$,
$$
\ell(x,\tilde z,\tilde a/h)h \geq \ell(x,\tilde z,a/h)h+\tilde w\cdot (\tilde a-a) + \frac{\gamma}{h}|\tilde a -a|_1\,|\tilde a-a|_\infty,
$$
which leads to
\begin{align*}
H(x,\tilde z,&\,E(x,\tilde z,\varphi,\hat \alpha), \tilde a)=
\ell(x,\tilde z,\tilde a/h)h + \tilde a \cdot E(x, \tilde  z,\varphi,\hat\alpha) \\
&\geq \ell(x,\tilde z,a/h)h + a\cdot E(x,\tilde z,\varphi,\hat\alpha) +[\tilde w +E(x, \tilde  z,\varphi,\hat\alpha)]\cdot (\tilde a-a)+\frac{\gamma}{h}|\tilde a -a|_1\,|\tilde a-a|_\infty \\
&= H(x,\tilde z,E(x, \tilde z,\varphi,\hat\alpha),a) + [\tilde w +E(x, \tilde  z,\varphi,\hat\alpha)]\cdot (\tilde a-a) + \frac{\gamma}{h} |\tilde a -a|_1\,|\tilde a -a|_\infty.
\end{align*}
Now we apply this to  $a=\hat\alpha(x,z)$ and $\tilde a=\hat\alpha(x,\tilde z)$. By the definition of $\hat\alpha(x,\tilde z)$ as the minimizer of $H(x,\tilde z,E(x,\tilde  z,\varphi,\hat \alpha),\cdot\,;\hat \alpha)$,
\begin{align*}
H(x,\tilde z,E(x,\tilde  z,\varphi,\hat \alpha),\hat\alpha(x,z)) &\geq H(x,\tilde z,E(x,\tilde  z,\varphi,\hat \alpha),\hat\alpha(x,\tilde z))\\
&\geq H(x,\tilde z,E(x,\tilde z,\varphi,\hat \alpha),\hat\alpha(x,z))  \\
&\quad +[\tilde w+ E(x,\tilde z,\varphi,\hat\alpha)] \cdot (\hat\alpha(x,\tilde z) - \hat \alpha(x,z)) \\
&\quad + \frac{\gamma}{h}|\hat\alpha(x,\tilde z) - \hat \alpha(x,z)|_1\, |\hat\alpha(x,\tilde z) - \hat \alpha(x,z)|_\infty.
\end{align*}
As $\hat\alpha(x,z)$ minimizes $H(x, z,E(x,z,\varphi,\hat \alpha),\,\cdot\,)$, by \cite[Corollary 3.68]{FOMO} there exists a $w\in\partial_a\ell(x,z,\hat\alpha(x,z)/h)$ such that 
$$
(w+ E(x,z,\varphi,\hat\alpha)) \cdot (\hat\alpha(x,\tilde z)-\hat\alpha(x,z)) \geq0.
$$
Combining the last two inequalities,
\begin{align*}
0&\geq [\tilde w+ E(x,\tilde z,\varphi,\hat \alpha)] \cdot (\hat\alpha(x,\tilde z) - \hat \alpha(x,z)) + \frac{\gamma}{h} |\hat\alpha(x,\tilde z) - \hat \alpha(x,z)|_1\, |\hat\alpha(x,\tilde z) - \hat\alpha(x,z)|_\infty\\
&\geq [\tilde w  -w + E(x,\tilde z,\varphi,\hat\alpha) -   E(x, z,\varphi,\hat\alpha) ] \cdot (\hat\alpha(x,\tilde z) - \hat \alpha(x,z)) +\frac{\gamma}{h}|\hat\alpha(x,\tilde z) - \hat \alpha(x,z)|_1\, |\hat\alpha(x,\tilde z) - \hat \alpha(x,z)|_\infty.
\end{align*}
By Lemma \ref{lem:lipschitz-multinom-II},
\begin{align*}
\sum_{y\in S(x)}\E| \varphi(y,Z'(x,\tilde z,\hat\alpha))- \varphi(y,Z'(x,z,\hat\alpha))| &\leq L_\varphi\ \E|Z'(x,\tilde z,\hat\alpha)-Z'(x,z,\hat\alpha)|_1\\
&\leq L_\varphi\ (1+\max_{x\in\cX} \Lip_1(\hat \alpha(x,\cdot)) )\,|\tilde z-z|_1.
\end{align*}
By H\"{o}lder's inequality,
\begin{align*}
\frac{\gamma}{h}|\hat\alpha(x,\tilde z) - \hat \alpha(x,z)|_1  &\leq \sum_{y\in S(x)} \bigl(|\tilde w(y)-w(y)| +
\E|\varphi(y,Z(x,\tilde z,\hat\alpha))-\varphi(y,Z(x,z,\hat\alpha))| \bigr)\\
&\leq \bigl(L_{\partial\ell} +  L_\varphi (1+\max_{x\in\cX} \, \Lip_1(\hat \alpha(x,\cdot))) \bigr) \,|\tilde z-z|_1.
\end{align*}
This implies
$$
(\gamma/h- L_\varphi) \max_{x\in\cX} \, \Lip_1(\hat\alpha(x,\cdot)) \leq L_{\partial\ell}+L_\varphi.
$$
Note that $\gamma/h-L_\varphi>0$ by the first part of the proof.
\end{proof}

Using Assumption \ref{asm:basic} \eqref{asm:basic-(iv)}, we know that there exists a $b_*<\infty$ such that the value function \eqref{eq:value-function-tagged-player} of the tagged player is bounded by $b_*(1+T)$, uniformly in the step size $h\in(0,\bar h)$ for any $\bar h<\infty$. Now note that any function $\varphi:\Sigma^{d-1}_N\to\R$ trivially satisfies
$$
|\varphi(\tilde z)-\varphi(z)|\leq 2N\|\varphi\|_\infty |\tilde z- z|_\infty, 
$$ 
so $\Lip_\infty (\varphi) \leq 2N\|\varphi\|_\infty.$ Together with $\Lip_1 (\varphi) \leq \Lip_\infty (\varphi) $ and Proposition \ref{prop:finite-player-one-step}, this implies the following result.

\begin{Prop}\label{prop:finite-player-h(N,T)}
Let Assumptions \ref{asm:basic} and \ref{asm:convex-cost} hold. Let $h>0$ be a time step and $T\in h\Z_+$ an arbitrary time horizon. If
$$
h< h_*(N,T):= \frac{\gamma}{2mNb_*(1+T)},
$$
then there exists a unique solution $v:\bar\cT\times\sX_N\to [0,\infty)$ to \eqref{eq:N-NLL} on $\bar\cT=\{0,h,2h,\ldots,T\}$.
\end{Prop}

This establishes Theorem \ref{thm:symmetric-Nash-system-and-sMPE} \eqref{thm:symmetric-Nash-system-and-sMPE-(ii)}.  In Section \ref{sec:short-time} we will see that there are time steps uniform in $N$ that ensure uniqueness in short time. 

The following example of a two-player game illustrates the necessity of the small-$h$ condition to ensure the existence and uniqueness of a MPE. We will analyze the mean-field version of this example in Subsection \ref{ssec:MF-one-step-ex}.

\begin{Ex}\label{ex:N-one-step}
{\rm
For $N=1$ and a (sufficiently small) thermal noise level $\sigma^2\geq0$, set
$$
\cX=\{0,1\},\quad A(h,x)=\{(a_0,a_1)\in\Sigma^1:a_0,a_1\geq\sigma^2h\},\quad \ell(x,a) = \frac{1}{2}a_{x+1}^2,\quad g(x,z)=  \chi_{x\neq z},
$$
for $(x,z,a)\in\cX^2\times \Sigma^1$. Addition is understood modulo $2$, and we observe that the cost function $g$ encourages players to be in the same state at terminal time while $\ell$ penalizes changing one's current state. In this two-state game, we can identify controls $\alpha:\cX^2\to\Sigma^1$ with the probability of changing, i.e., write $\alpha(x,z)=\alpha_{x+1}(x,z)$. Then, admissible controls are vectors $\alpha(x,z)\in \tilde A(h)=:[\sigma^2h,1-\sigma^2h]$. Given a control $\beta(x,z)$ of the other player, the tagged player solves
$$
\inf_{\alpha:\,\cX^2\to \tilde A(h)} \alpha(x,z)^2/(2h)+\E^{\alpha\otimes\beta}[g(X_h,Z_h)\,|\,(X_0,Z_0)=(x,z)],
$$
for every $(x,z)\in\cX^2$. A direct computation shows that the NLL equation characterizing symmetric MPE reduces to the following system of equations:
$$
\begin{cases}
\alpha(0,0) = h[2\alpha(0,0)-1]\lor(\sigma^2h)\land (1-\sigma^2h),\\
\alpha(0,1) = h[1-2\alpha(1,0)]\lor(\sigma^2h)\land (1-\sigma^2h),\\
\alpha(1,0) = h[1-2\alpha(0,1)]\lor(\sigma^2h)\land (1-\sigma^2h),
\end{cases}
$$
and $\alpha(1,1)=\alpha(0,0)$. One easily observes the necessity for small $h$ to ensure the unique solvability of this set of equations. Indeed, if $\sigma^2=0$, then the first equation possesses multiple solutions for $h\geq1$, two of them being $\alpha(0,0)\in\{0,1\}$. If $\sigma^2>0$, another direct computation shows the existence of three solutions if $h\in(h_*,h^*)$ where $h_*:=(1+\sigma^2 -\sqrt{1-(6-\sigma^2) \sigma^2})/(4 \sigma^2)$ and $h^*:=(1+\sigma^2 +\sqrt{1-(6-\sigma^2) \sigma^2})/(4 \sigma^2)$. This assumes that $\sigma^2\leq 3-2\sqrt{2}$, so that the square root is well-defined.
}
\end{Ex}

\section{The mean-field problem} \label{sec:mean-field-problem}

\subsection{The mean-field game system}

We continue to establish the existence result in Proposition \ref{prop:MFG-NE and MF-NLL} \eqref{prop:MFG-NE and MF-NLL-(i)}. The proof of \ref{prop:MFG-NE and MF-NLL} \eqref{prop:MFG-NE and MF-NLL-(ii)} follows by inspection and hence we omit it.

\begin{proof}[Proof of Proposition \ref{prop:MFG-NE and MF-NLL} \eqref{prop:MFG-NE and MF-NLL-(i)}]
Set $\cS := (\Sigma^{d-1})^{\cT(t,T)}$ and define a map $\Phi:\cS\to \cS$ via the following procedure. For $\boldsymbol\mu\in \cS$, let $(v,\alpha)$ solve the dynamic programming equation for $t\leq s<T$:
$$
\begin{cases}
v(s,x)= \underset{a\in A(h,x)}{\inf} \ H(x,\mu_s,v(s+h),a),\\
\alpha(s,x) = \underset{a\in  A(h,x) }{\mathrm{arg\,min}} \ H(x,\mu_s,v(s+h),a),\\
v(T,x) = g(x,\mu_T).
\end{cases}
$$
Then, set
$$\Phi_{t}(\boldsymbol\mu):=\mu,\quad 
\Phi_{s+h}(\boldsymbol\mu) := \Bigl(\sum_{y\in\cX}\mu_s(y)\alpha_x(s,y) \Bigr)_{x\in\cX},\quad t\leq s<T.
$$
With this construction, it is clear that solutions to the \ref{eq:MFG-system} correspond to fixed points of the map $\Phi$, and hence by Remark \ref{rmk:MFG-system-correspondence}, correspond to mean-field game Nash equilibria.  Berge's maximum theorem implies that $\Phi$ is a continuous map. The claim follows by Brouwer's fixed-point theorem.
\end{proof}

\subsection{The mean-field Nash-Lasry-Lions equation}

We now discuss \eqref{eq:MF-NLL} and first recall that any solution to \eqref{eq:MF-NLL}, regardless of its regularity in the $\mu$ variable, defines a solution to the \ref{eq:MFG-system} for any initial condition $(t,\mu) \in \cT\times\Sigma^{d-1}$, and hence gives rise to solutions to the mean-field game, see Proposition \ref{prop:MFG-NE and MF-NLL} \eqref{prop:MFG-NE and MF-NLL-(ii)}. 

In Section \ref{sec:short-time}, we show the existence and uniqueness of the solution of \eqref{eq:MF-NLL} in short time. To this end, we again turn to the one-step problems. 

\begin{Prop}\label{prop:MF-one-step}
Let Assumptions \ref{asm:basic} and \ref{asm:convex-cost} hold. Then, for any Lipschitz continuous function $\varphi:\sX\to\R$, there exists a unique map $\hat\alpha:\sX\to\Sigma^{d-1}$ that satisfies
$$
\hat \alpha(x,\mu) = \underset{a\in  A(h,x) }{\mathrm{arg\,min}} \, H(x,\mu,\varphi(\cdot,\mu\cdot \hat \alpha(\cdot,\mu)),a),\quad \forall (x,\mu)\in\sX,
$$
provided that $h<\gamma/L_\varphi$, where $L_\varphi:=m\,\max_{x\in\cX}\Lip_1(\varphi(x,\cdot))$. In addition, let Assumption \ref{asm:lipschitz-cost} hold. Then, $\hat \alpha$ is Lipschitz continuous in $\mu$ with 
$$
\max_{x\in\cX} \, \Lip_1(\hat\alpha(x,\cdot)) \leq \frac{h(L_{\partial\ell}+L_\varphi)}{\gamma-hL_\varphi},
$$
where $\gamma$ and $L_{\partial\ell}$ are constants from Assumptions \ref{asm:convex-cost} and \ref{asm:lipschitz-cost}, respectively.
\end{Prop}

\begin{proof}
For any fixed $\mu\in\Sigma^{d-1}$, we consider the following finite-dimensional fixed-point problem. Define
$$
\cS:=\{\alpha:\cX\to \Sigma^{d-1}\,:\,\alpha_y(x)=0 \ \forall y\notin  S(x)\},\quad d_\cS(\alpha,\tilde\alpha) := \max_{x\in\cX} \sum_{y\in S(x)}|\alpha_y(x)-\tilde\alpha_y(x)|,\quad \alpha,\tilde\alpha\in\cS.
$$
Next, define $\Phi:\cS\to \cS$ by
$$
\Phi_x(\alpha) := \underset{ a\in  A(h,x) }{\mathrm{arg\,min}} \,H(x,\mu,\varphi(\cdot,\mu\cdot \alpha),a) ,\quad \text{and}\quad  \Phi(\alpha) :=(\Phi_x(\alpha))_{x\in\cX}.
$$
We first note that, if $\mu\in\Sigma^{d-1}$ is a row vector and $A\in\R^{d\times d}$, then $|\mu\cdot A|_1\leq \max_{1\leq i \leq d}|A_{i,:}|_1$.
Define a function $f:\R^d\to[0,\infty]$ by $f(a):=\ell(x,\mu,a/h)h$ for $a\in[0,\infty)^d$ and $+\infty$ otherwise and set $C:=A(h,x)$. Using Assumption \ref{asm:convex-cost}, we again see $f$ satisfies the conditions of Lemma \ref{lem:convex-conjugates}.
Observe that
$$
\Phi_{x}( \alpha)  =  \underset{a\in C}{\mathrm{arg\,max}} \, \Bigl\{\sum_{y\in S(x)}\! a_y \, (-\varphi(y,\mu\cdot\alpha)) -f(a) \Bigr\}.
$$
Analogously to the proof of Proposition \ref{prop:finite-player-one-step}, Lemma \ref{lem:convex-conjugates} implies 
$$
|\Phi_x(\alpha)-\Phi_x(\tilde \alpha)|_1 \leq \frac{h}{2\gamma} \sum_{y\in S(x)}|\varphi(y ,\mu\cdot\alpha)-\varphi(y,\mu\cdot\tilde \alpha)| \leq \frac{hL_\varphi}{2\gamma} \, |\mu\cdot(\alpha-\tilde \alpha)|_1 \leq \frac{hL_\varphi}{2\gamma} \, \max_{x\in\cX}|\alpha(x)-\tilde \alpha(x)|_1,
$$
so that $\Phi$ is a contraction since $h<\gamma/L_\varphi$, where $L_\varphi =m \max_{x\in\cX}\Lip_1(\varphi(x,\cdot))$.
By Banach's fixed-point theorem, there exists a unique fixed point in $\cS$. Since $\mu\in\Sigma^{d-1}$ was arbitrary, this construction yields a unique fixed point $\hat\alpha(x,\mu)$. The Lipschitz constant of $\hat\alpha$ can be deduced similarly to the finite-player case, see the proof of Proposition \ref{prop:finite-player-one-step}.
\end{proof}

We proceed with an explicit example that illustrates the necessity for the smallness condition on the time step.

\begin{Rmk}\label{rmk:monotonicity}
{\rm
Our paper does not assume any monotonicity conditions on the cost functions, and the next example shows that we cannot expect global well-posedness. Our existence and uniqueness result in Section \ref{sec:short-time} relies on the smallness of both the time step $h$ and the time horizon $T$. Global existence and uniqueness for \eqref{eq:MF-NLL} can for example be established under the Lasry-Lions monotonicity condition. This condition allows to construct a time-consistent family of Nash equilibria, which is related to the concept of pasting equilibria by Dianetti, Nendel, Tangpi and Wang \cite{dianetti2024pasting}. 
}
\end{Rmk}

\subsection{One-step example} \label{ssec:MF-one-step-ex}

This section presents an example of a mean-field game in which uniqueness is lost in one step when the time step $h$ exceeds a critical value.  For a (sufficiently small) thermal noise level $\sigma^2\geq0$, set  
$$
\cX=\{0,1\}, \quad \tilde A(h)=[\sigma^2h,1-\sigma^2h],\quad \ell(x,a) = \frac{1}{2}a^2,\quad g(x,\mu)=\mu\,\chi_{x=0}+ (1-\mu)\,\chi_{x=1},
$$
for $(x,a,\mu)\in\cX\times \tilde A(h)\times[0,1]$. We assume $\sigma^2h\leq1/2$, so that $\tilde A(h)\neq\emptyset$. As in the two-player game in Example \ref{ex:N-one-step}, addition on $\cX$ is defined modulo $2$ and we identify controls $\alpha(x)=(\alpha_0(x),\alpha_1(x))$ with the probability $\alpha_{x+1}(x)\in \tilde A(h)$ of \emph{changing} the current state. Similarly, we only keep track of the fraction $\mu\in[0,1]$ of the population in state $1$. Given a flow $\mu=(\mu_0,\mu_h)\in[0,1]^2$ of probabilities representing the fractions of players in state $1$, a typical player solves
$$
\inf_{\alpha:\,\cX\to \tilde A(h)} \alpha(x)^2/(2h)+\E^{\alpha}[g(X_h,\mu_h)\,|\,X_0=x],
$$
for every $x\in\cX$. The cost functional is chosen in a way that leads to synchronization with the population. Given the initial distribution $\mu_0$, mean-field game equilibria are characterized by the following system:
$$
\begin{cases}
\alpha(x) = \underset{a\in A(h)}{\mathrm{arg\,min}} \,\{a^2/(2h)+ (1-a) g(x,\mu_h)+a g(x+1,\mu_h)\},\quad x\in\cX,\\
\mu_h = \alpha(0)(1-\mu_0) + (1- \alpha(1))\mu_0.
\end{cases}
$$
Computing $\alpha$ in terms of $\mu_h$ leads to the following one-dimensional fixed-point problem for the distribution at time $t=h$:
$$
\mu_h = \alpha^{\mu_h}(0)(1-\mu_0)+(1-\alpha^{\mu_h}(1))\mu_0,
$$
with $\alpha^{\mu_h}(0):=[h(2\mu_h-1)]\lor(\sigma^2h)\land(1-\sigma^2h)$ and $\alpha^{\mu_h}(1):=[h(1-2\mu_h)]\lor(\sigma^2h)\land(1-\sigma^2h)$. This is a fixed-point equation for the Nash equilibrium distribution $\mu_h\in[0,1]$. If $\sigma^2=0$ and $h\geq1$, then there are three equilibria, $\mu_h\in\{0,1/2,1\}$, starting from the uniform distribution. If $\sigma^2>0$, we have the following result, which follows by inspection of the above fixed-point equation, see Figure \ref{fig:1-step-MFG-Kuramoto} for an illustration.

\begin{Lem}
Let $\mu_0=1/2$ and $0< \sigma^2< 3-2\sqrt{2}$. Define $h_*:=(1+\sigma^2 -\sqrt{1-(6-\sigma^2) \sigma^2})/(4 \sigma^2)$ and $h^*:=(1+\sigma^2 +\sqrt{1-(6-\sigma^2) \sigma^2})/(4 \sigma^2)$.  
\begin{enumerate}[(i)]
\item For any $h>0$, the uniform distribution $\mu_h=1/2$ is a Nash equilibrium.
\item If $h\in(0,h_*)\cup(h^*,1/(2\sigma^2))$, then the only NE is $\mu_h=1/2$.
\item If $h\in\{h_*,h^*\}$, then there are exactly three NE, including $\mu_h=1/2.$
\item If $h\in(h_*,h^*)$, then there are exactly five NE, including $\mu_h=1/2.$ 
\end{enumerate}
\end{Lem}

\begin{figure}[h]
    \begin{subfigure}{0.3\textwidth}
        \centering
        \includegraphics[width=.9\linewidth]{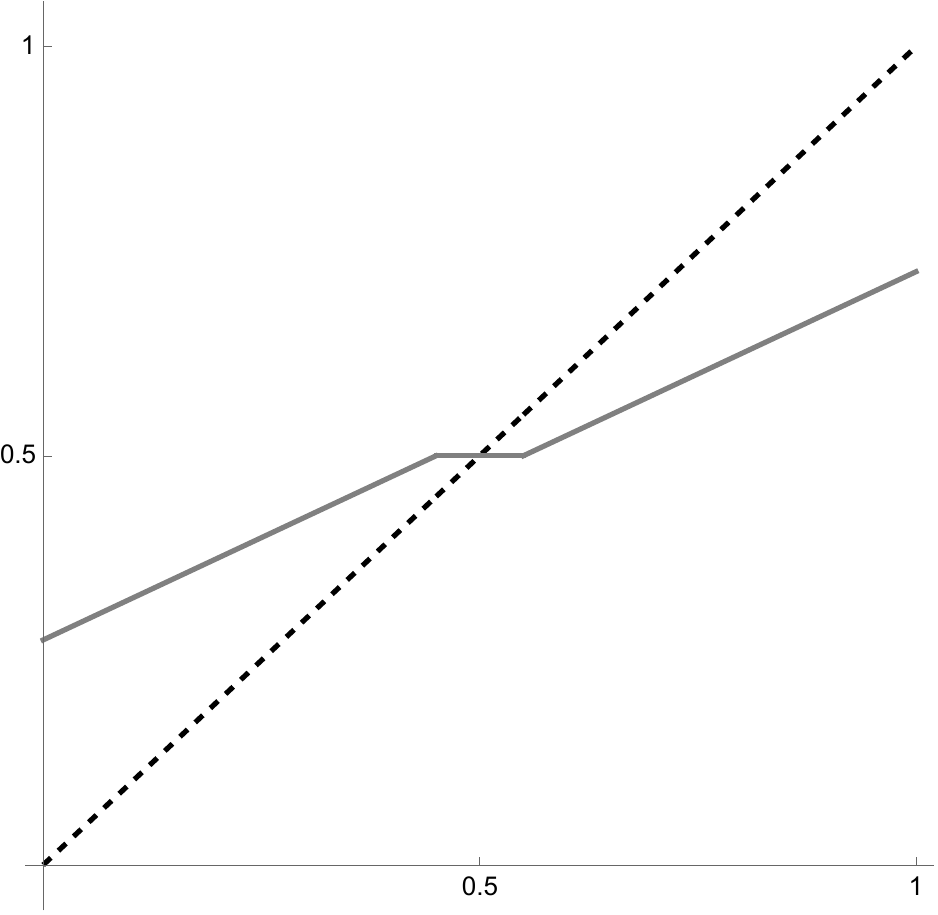}
        \caption{$h<h_*$}
    \end{subfigure}
    \hspace{5pt}
    \begin{subfigure}{0.3\textwidth}
        \centering
        \includegraphics[width=.9\linewidth]{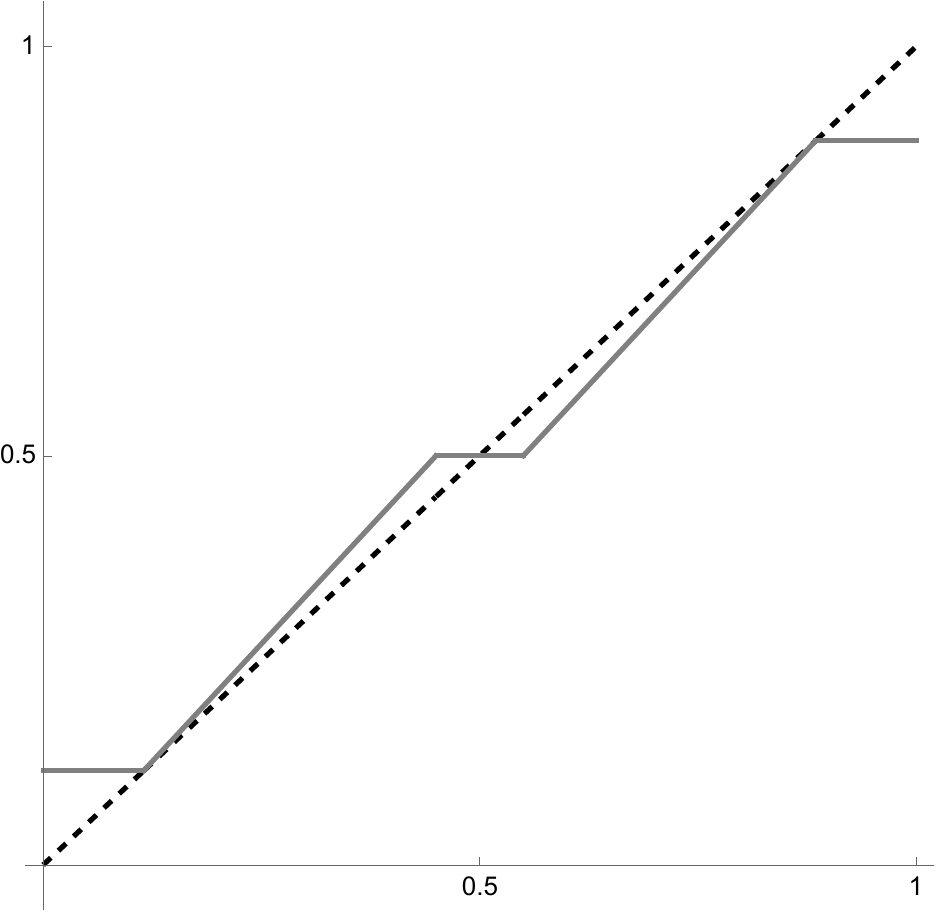}
        \caption{$h=h_*$}
    \end{subfigure}
    \hspace{5pt}
    \begin{subfigure}{0.3\textwidth}
        \centering
        \includegraphics[width=.9\linewidth]{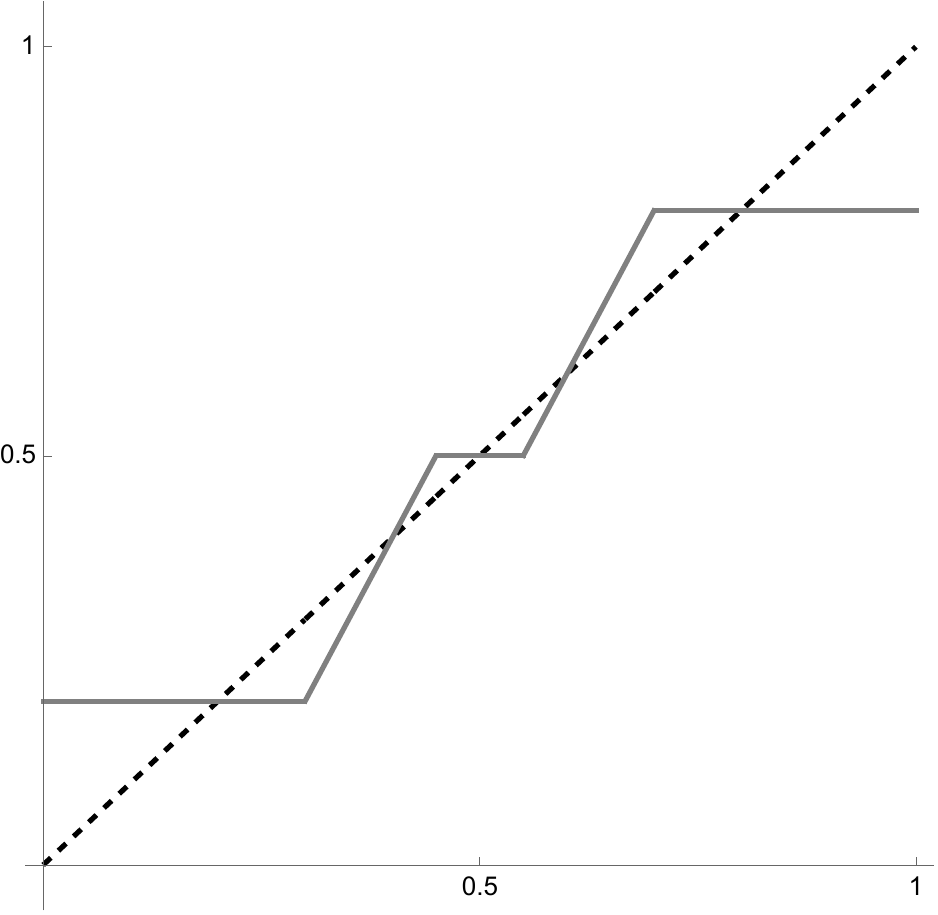}
        \caption{$h\in(h_*,h^*)$}
    \end{subfigure}
    \caption{Nash equilibria starting from $\mu_0=1/2$ when $0<\sigma^2<3-2\sqrt{2}$.}
    \label{fig:1-step-MFG-Kuramoto}
\end{figure}

Clearly, in any one-step mean-field game, solutions to the Nash-Lasry-Lions equation are in one-to-one correspondence with collections $(\mu_0,\mu_h^{\mu_0})_{\mu_0\in\Sigma^{d-1}}$ of Nash equilibria, parametrized by the initial condition $\mu_0$. This observation implies that, in the example above, for $h\in[h_*,h^*]$, uniqueness of solutions to  \eqref{eq:MF-NLL} is lost.

\section{Short-time results} \label{sec:short-time}

Theorem \ref{thm:symmetric-Nash-system-and-sMPE} establishes the existence and uniqueness of the solution to \eqref{eq:N-NLL} for a time step $h<h_*(N,T)$ that depends on both the number of players and the time horizon of the game. In fact, as stated in Theorem \ref{thm:well-posedness-MF-NLL}, there exist time steps that guarantee uniqueness for any finite number of players as well as for the mean-field limit. We proceed to establish this result.

\begin{Lem}
Let Assumptions \ref{asm:basic}, \ref{asm:convex-cost} and \ref{asm:lipschitz-cost} hold. 
\begin{enumerate}[(a)]
\item (Finite-player case). 
Let $\varphi:\sX_N\to\R$ be given. Let $\hat \alpha$ be as in Proposition \ref{prop:finite-player-one-step}, and define
$$
\psi(x,z) = H(x,z,E(x,z,\varphi,\hat \alpha),\hat\alpha(x,z)), \quad (x,z)\in\sX_N.
$$
\item (Mean-field case). 
Let $\varphi:\sX\to\R$ be given. Let $\hat \alpha$ be as in Proposition \ref{prop:MF-one-step}, and define 
$$
\psi(x,\mu)= H(x,\mu,\varphi(\cdot,\mu\cdot\hat\alpha(\cdot,\mu)),\hat\alpha(x,z)), \quad (x,\mu)\in\sX.
$$
\end{enumerate}
In both cases, we have 
\be\label{eq:Lipschitz-estimate-value-func}
L_{\psi} \leq hm\,L_\ell + L_{\varphi} \, \Bigl(1+ \frac{h(L_{\partial\ell}+L_{\varphi})}{\gamma - hL_\varphi}\Bigr),
\ee
where $L_{\psi}=m\,\max_{x\in\cX} \,\Lip_1(\psi(x,\cdot))$ and  $L_{\varphi}=m\,\max_{x\in\cX} \, \Lip_1(\varphi(x,\cdot))$.
\end{Lem}

\begin{proof}
(a) By Lemma \ref{lem:lipschitz-multinom-II}, for $x\in\cX$ and $z,\tilde z\in\Sigma^{d-1}_N$,
\begin{align*}
|\psi(x, z)-\psi(x, \tilde z)| &\leq \sup_{a\in  A(h,x) } |(\ell(x,z,a/h)-\ell(x, \tilde z,a/h))h +  a\cdot (E(x,z,\varphi,\hat \alpha)-E(x, \tilde z,\varphi,\hat \alpha))|\\
&\leq \bigl[h L_\ell+ L_{\varphi}/m \, (1+ L_{\hat \alpha}) \bigr]\, |z-\tilde z|_1,
\end{align*}
where $L_{\hat \alpha} = \max_{x\in\cX}\,\Lip_1(\hat \alpha(x,\cdot))$. Using the Lipschitz estimate from Proposition \ref{prop:finite-player-one-step}, we obtain
$$
L_{\psi}\leq hm\, L_\ell+ L_{\varphi}  \, (1+L_{\hat\alpha}) \leq hm\, L_\ell+ L_{\varphi} \, \Bigl(1+ \frac{h(L_{\partial\ell}+ L_{\varphi})}{\gamma - h L_\varphi}\Bigr).
$$

(b) For $\mu,\tilde\mu\in\Sigma^{d-1}$ and $a\in A(h,x)$,
\begin{align*}
|a\cdot(\varphi(\cdot,\mu \cdot\hat \alpha (\cdot,\mu))-\varphi(\cdot,\tilde \mu \cdot\hat \alpha (\cdot,\tilde\mu)))|&\leq L_{\varphi}/m\,|\mu \cdot\hat \alpha (\cdot,\mu)-\tilde \mu \cdot\hat \alpha (\cdot,\tilde\mu)|_1\\
&=L_{\varphi}/m\, |(\mu-\tilde\mu)\cdot \hat \alpha (\cdot,\mu) + \tilde\mu\cdot( \hat \alpha (\cdot,\mu)-\hat \alpha(\cdot,\tilde\mu))|_1\\
&\leq L_{\varphi}/m(1+L_{\hat \alpha }) |\mu-\tilde \mu|_1,
\end{align*}
where $L_{ \hat \alpha} =\max_{x\in\cX} \,\Lip_1( \hat \alpha (x,\cdot))$.
Therefore,
\begin{align*}
|\psi(x,\mu)-\psi(x,\tilde \mu)| &\leq \sup_{a\in  A(h,x) } |(\ell(x,\mu,a/h)-\ell(x,\tilde \mu,a/h))h +  a\cdot(\varphi(\cdot,\mu \cdot\hat \alpha (\cdot,\mu))-\varphi(\cdot,\tilde \mu \cdot\hat \alpha (\cdot,\tilde\mu)))|\\
&\leq \bigl[h L_\ell+ L_{\varphi}/m(1+L_{ \hat \alpha })\bigr]\, |\mu-\tilde \mu|_1,
\end{align*}
and we complete the proof using the Lipschitz estimate from Proposition \ref{prop:MF-one-step}, analogous to (a).
\end{proof}

The following proposition establishes Theorem \ref{thm:well-posedness-MF-NLL}.

\begin{Prop}\label{prop:short-time-well-posedness}
Let Assumptions \ref{asm:basic}, \ref{asm:convex-cost} and \ref{asm:lipschitz-cost} hold. 
There exist constants $h_*,T_*>0$ depending only on the cost functions $\ell,g$ and $m$, such that, for all $h>0$ and $K\geq1$, the following holds: If $h<h_*$ and $Kh<T_*$, then there exists a unique solution $(v^{(\infty)},\alpha^{(\infty)})$ to \eqref{eq:MF-NLL} and a unique solution $(v^{(N)},\alpha^{(N)})$ to \eqref{eq:N-NLL} on $\bar\cT=\{0,\ldots,Kh\}$ for every $N\geq1$. Furthermore, 
\begin{equation}
    \label{eq:prop-short-time-well-posedness}
\sup_{N\in\N\cup\{\infty\}}\max_{(t,x)\in\bar\cT\times\cX} \, \big[\Lip_1(v^{(N)}(t,x,\cdot))+\Lip_1(\alpha^{(N)}(t,x,\cdot))\big]<\infty.
\end{equation}
\end{Prop}

\begin{proof}
\emph{Step 1.} We iterate the estimate \eqref{eq:Lipschitz-estimate-value-func}, which is valid for both the finite-player and the mean-field case. We reverse time, set $L(0):= m\,L_g := m\, \max_{x\in\cX}\,\Lip_1(g(x,\cdot))$, and recursively define
$$
L(k) = L(k-1)+ h \Bigl(m L_\ell +  L(k-1) \frac{L_{\partial\ell}+L(k-1)}{\gamma-h L(k-1)}\Bigr),\quad k\geq 1.
$$
Fix $M>L(0)$, $h\leq\gamma/M$ and set $\tilde M := mL_\ell + M(L_{\partial\ell}+M)/(\gamma-hM)$. We claim that, for any $K\geq0$ that satisfies $
Kh\leq (M-L(0))/\tilde M$, we have $L(k)\leq M$ for all $k=0,\ldots,K$. Indeed, we may proceed by induction on $K$. For $K=0$, we have $L(0)\leq M$ by assumption. Let $K\geq1$. If $L(0),\ldots,L(K-1)\leq M$, then 
\begin{align*}
L(K)&=L(0)+\sum_{k=1}^K (L(k)-L(k-1))\\
&\leq L(0)+ h \sum_{k=1}^K \Bigl(mL_\ell +  L(k-1) \frac{L_{\partial\ell}+L(k-1)}{\gamma-h L(k-1)}\Bigr)\leq L(0) + Kh\, \tilde M \leq M,
\end{align*}
where the first inequality follows from the induction hypothesis, the second from the choice of 
$h$ and $\tilde M$, and the third follows from the choice of the product $Kh$. \\

\emph{Step 2.} In light of the previous step, we may choose any $M>L(0)$ and define $h_*:=\gamma/M$ and $T_*:=(M-L(0))/\tilde M$. Note that, in particular, we have $h<\gamma/L(k)$ for all $k=0,\ldots,K-1$. By Propositions \ref{prop:finite-player-one-step} and \ref{prop:MF-one-step}, we can recursively solve the equations \eqref{eq:N-NLL} and \eqref{eq:MF-NLL} on $\{0,\ldots,Kh\}$.  Finally, note that the Lipschitz constants of the value functions are bounded by $M$, which implies a bound for the Lipschitz constants of the optimal controls as well.
\end{proof}

\section{Mean-field convergence} \label{sec:mean-field-convergence}

We proceed to establish Theorem \ref{thm:N-convergence}, which asserts the convergence of the solutions to \eqref{eq:N-NLL} to the solution of the mean-field equation \eqref{eq:MF-NLL} as $N\to\infty$ in short time. We start by recording the following consequence of the law of large numbers.

\begin{Lem}\label{lem:LLN}
Let $(\varphi^{(N)},\alpha^{(N)}):\sX\to\R\times\Sigma^{d-1}$, $N\geq1$, be a sequence of functions converging uniformly to some $(\varphi^{(\infty)},\alpha^{(\infty)}):\sX\to\R\times\Sigma^{d-1}$. Assume that $\sup_{N\geq1}\Lip(\varphi^{(N)}(x,\cdot))<\infty$ for all $x\in\cX$. Let $(z^{(N)})_{N\geq1}$ be a sequence with $z^{(N)}\in\Sigma^{d-1}_N$ converging to some $\mu\in\Sigma^{d-1}$. Then, for any $x\in\cX$,
$$
\E[\varphi^{(N)}(x,Z'(x,z^{(N)},\alpha^{(N)}))] \to \varphi^{(\infty)}\Big(x,\sum_{y\in\cX} \mu(y)\alpha^{(\infty)}(y,\mu)\Big),\quad \text{as} \quad N\to\infty.
$$
\end{Lem}
\begin{proof}
By the law of large numbers, as $N\to\infty$,
$$
Z'_j(x,z^{(N)},\alpha^{(N)}) = \sum_{y\in\cX} Z'_j(x,y,z^{(N)},\alpha^{(N)})\to\sum_{y\in\cX} \mu_y\, \alpha^{(\infty)}_j(y,\mu) =: Z^{(\infty)}_j,
$$
for each $1\leq j\leq d$ almost surely. We set $Z^{(\infty)}:=(Z^{(\infty)}_1,\ldots,Z^{(\infty)}_d)$. By the triangle inequality,
\begin{align*}
|\E[\varphi^{(N)}(x,Z'(x,z^{(N)},\alpha^{(N)}))] - \varphi^{(\infty)}(x,Z^{(\infty)})| &\leq \E|\varphi^{(N)}(x,Z'(x,z^{(N)},\alpha^{(N)}))-\varphi^{(N)}(x,Z^{(\infty)})|\\
&\quad + |\varphi^{(N)}(x,Z^{(\infty)})-\varphi^{(\infty)}(x,Z^{(\infty)})|.
\end{align*}
The second term converges to zero by assumption. The first term can be estimated as follows:
\begin{align*}
\E|\varphi^{(N)}(x,Z'(x,z^{(N)},\alpha^{(N)}))-\varphi^{(N)}(x,Z^{(\infty)})|&\leq\Lip(\varphi^{(N)}(x,\cdot)) \,\E|Z'(x,z^{(N)},\alpha^{(N)})-Z^{(\infty)}| \to 0.\qedhere
\end{align*}
\end{proof}

\begin{proof}[Proof of Theorem \ref{thm:N-convergence}]
\emph{Step 1.} Let $h_*,T_*$ be as in Proposition \ref{prop:short-time-well-posedness}. By \eqref{eq:prop-short-time-well-posedness}, the solutions $(v^{(N)},\alpha^{(N)})_{N\geq1}$ of \eqref{eq:N-NLL} form a uniformly bounded equicontinuous sequence of continuous functions from $\Sigma^{d-1}$ to $[0,\infty)^{\bar\cT\times \cX} \times (\Sigma^{d-1})^{\cT\times\cX}$. Here, we linearly interpolated in the $z$-variable. Hence, by the Arzel\`{a}-Ascoli theorem, there exists a $(\tilde v^{(\infty)}, \tilde \alpha^{(\infty)})$ such that, along a subsequence, $(v^{(N)},\alpha^{(N)})\to (\tilde v^{(\infty)},\tilde\alpha^{(\infty)})$ uniformly as $N\to\infty$. We replace the original sequence by this subsequence.

\emph{Step 2.} We now claim that 
$$
\tilde \alpha^{(\infty)}(t,x,\mu) = \underset{a\in  A(h,x) }{\mathrm{arg\,min}} \, H(x,\mu,\tilde v^{(\infty)}(t+h,\cdot,\mu\cdot \tilde\alpha^{(\infty)}(t,\cdot,\mu)),a),\quad (t,x,\mu)\in\cT\times\sX.
$$
Indeed, for every $N\geq1$ and $(x,z,a)\in\sX_N \times  A(h,x)$, we have
$$
H(x,z,E(x,z,v^{(N)}(t+h),\alpha^{(N)}(t)),\alpha^{(N)}(t,x,z))\leq H(x,z,E(x,z,v^{(N)}(t+h),\alpha_N(t)),a).
$$
For a given $\mu\in\Sigma^{d-1}$, choose $z^{(N)}\in\sX_N$ such that $z^{(N)}\to\mu$ as $N\to\infty$. By Lemma \ref{lem:LLN} and continuity of $\ell$, 
$$
H(x,\mu,\tilde v^{(\infty)}(t+h,\cdot,\mu\cdot\tilde\alpha^{(\infty)}(t,\cdot,\mu)),\tilde \alpha^{(\infty)}(t,x,\mu))\leq H(x,\mu,\tilde v^{(\infty)}(t+h,\cdot,\mu\cdot\tilde\alpha^{(\infty)}(t,\cdot,\mu)),a)
$$
for $(t,x,\mu,a)\in\cT\times\sX\times A(h,x)$. Hence, $(\tilde v^{(\infty)},\tilde \alpha^{(\infty)})$ is a solution of \eqref{eq:MF-NLL} on $\bar\cT$, which has a unique solution $(V,\alpha)$. This argument shows that any subsequence of $(v^{(N)},\alpha^{(N)})_{N\geq1}$ has a further subsequence that converges to $(V,\alpha)$ uniformly, hence the entire sequence converges to $(V,\alpha)$ uniformly.
\end{proof}

\section{Continuous-time limit}\label{sec:continuous-limit}

This section proves Theorem \ref{thm:continuous-convergence}. Recall the definitions of $S(x)$ and $A(h,x)$ in \eqref{eq:cts-limit-adm-prob}. We again set $m:=\max_{x\in\cX}|S(x)|$.

\begin{Lem}\label{lem:cts-time-limit}
Let $(\alpha^{(K)},\alpha^{(\infty)}):\sX_N\to \Sigma^{d-1}\times\R^d$, $K\geq1$, be functions such that $\mathfrak{a}^{(K)}/h^{(K)}\to \alpha^{(\infty)}$ as $K\to\infty$, where $\mathfrak{a}^{(K)}(x,z):=\alpha^{(K)}(x,z)-e_x$ for $(x,z)\in\sX_N$. Then, for all $(x,z,z')\in\sX_N\times\Sigma^{d-1}_N$,
$$
\lim_{K\to\infty}\frac{1}{h^{(K)}} \P[Z'(x,z,\alpha^{(K)})= z'] 
=
\begin{cases}
\infty,& z'=z,\\
Nz_y\, \alpha^{(\infty)}_w(y,z+e_{xy}),& \text{if} \ \exists w\neq y: z'=z+e_{wy}\\
0,&\text{otherwise},
\end{cases}
$$
and, in fact, $\lim_{K\to\infty }\P[Z'(x,z,\alpha^{(K)})= z]=1$.
\end{Lem}

\begin{proof}
Fix $(x,z)\in\sX_N$. For any state $y\in\cX$ and any next-step distribution $s\in\Sigma^{d-1}_{Nz_y}$ of the untagged players in state $y$,
$$
\P[Z'(x,y,z,\alpha^{(K)})=s]
= \frac{(Nz_y)!}{(Ns_1)!\cdots (Ns_d)!}  \, [1+\mathfrak{a}_y^{(K)}(y,z+e_{xy})]^{Ns_y}\prod_{w\neq y}[\mathfrak{a}^{(K)}_w(y,z+e_{xy})]^{Ns_w} .
$$
By inspection of this formula, we deduce the following:
\begin{enumerate}[(1)]
\item if $s=z_ye_y$, then
$$
\P[Z'(x,y,z,\alpha^{(K)})=z_ye_y] = [1+\mathfrak{a}^{(K)}_y(y,z+e_{xy})]^{Nz_y} \to 1,\quad \text{as} \ K\to\infty,
$$
using $\mathfrak{a}^{(K)}_y(y,z+e_{xy})\to0$;
\item if there exists a $w\in\cX\setminus\{y\}$ such that $s=z_ye_y+e_{wy}$, then 
$$
\frac{1}{h^{(K)}} \P[Z'(x,y,z,\alpha^{(K)})=z_ye_y+e_{wy}]
= \frac{(Nz_y)!}{(Nz_y-1)!}   \, [1+\mathfrak{a}_y^{(K)}(y,z+e_{xy})]^{Nz_y-1}\frac{\mathfrak{a}^{(K)}_w(y,z+e_{xy})}{h^{(K)}},
$$
which converges to $Nz_y \,\alpha^{(\infty)}_w(y,z+e_{xy})$;
\item in all other cases, 
$$
\frac{1}{h^{(K)}}\P[Z'(x,y,z,\alpha^{(K)})=s]
= \frac{(Nz_y)!}{(Ns_1)!\cdots (Ns_d)!}  \, [1+\mathfrak{a}_y^{(K)}(y,z+e_{xy})]^{Ns_y}
\frac{1}{h^{(K)}}\prod_{i\neq y}[\mathfrak{a}^{(K)}_i(y,z+e_{xy})]^{Ns_i} 
$$
converges to zero as $\mathfrak{a}_i^{(K)}(y,z+e_{xy})/h^{(K)}$ stays bounded as $K\to\infty$ for all $i$ and there are at least two states $i\neq y$ present in the product.
\end{enumerate}

Next, let $z'\neq z$ in $\Sigma^{d-1}_N$ be given and set $\Sigma^{d-1}_0:=\{0_{\R^d}\}$. We let $\Sigma(z')$ denote the set of all $s=(s(1),\ldots,s(d))$ such that $s(j)\in\Sigma^{d-1}_{Nz_y}$ and $z' = s(1)+\cdots+s(d)$. Then, the claim follows from the previous observation, together with the explicit formula
\begin{align*}
\P[Z'(x,z,\alpha^{(K)})=z']  &=\P\Bigl[\sum_{y\in\cX} Z'(x,y,z,\alpha^{(K)})=z'\Bigr] =\sum_{s\in\Sigma(z')} \prod_{y\in\cX}\P[Z'(x,y,z,\alpha^{(K)})=s(y)].\qedhere
\end{align*}
\end{proof}

\begin{proof}[Proof of Theorem \ref{thm:continuous-convergence}]

For each $K\geq1$, set $h^{(K)}:=T/K$ and define the discretization $0=s_0^{(K)}<s_1^{(K)}<\cdots<s_k^{(K)}=T$ of the time interval $[0,T]$ by $s^{(K)}_k:=kT/K$. For each $K\geq1$, let $(v^{(K)},\alpha^{(K)})$ solve \eqref{eq:N-NLL}, which is uniquely defined for sufficiently large $K$, say $K\geq K_0$. We extend these functions to $[0,T]\times\sX_N$ by linear interpolation in the time variable and set 
$$
\mathfrak{a}^{(K)}(t,x,z) := \alpha^{(K)}(t,x,z)-e_x\in\R^d,\quad (t,x,z)\in [0,T]\times\sX_N.
$$ 
In the continuous-time problem, let $v_{\mathrm{cts}}:[0,T]\times\sX_N\to[0,\infty)$ be the unique solution to \eqref{eq:cts-time-finite-player-NLL} and let $\alpha_{\mathrm{cts}}(t,x,z)=\hat\alpha(x,z,\Delta_x v_{\mathrm{cts}}(t,\cdot,z))$.\\

\emph{Step 1.} We claim that $(v^{(K)},\mathfrak{a}^{(K)}/h^{(K)})_{K\geq K_0}$ is a uniformly bounded, equicontinuous family of functions, seen as functions from $[0,T]\times\sX_N$ to $[0,\infty)\times \R^d$.

\emph{Uniform boundedness.} By Assumption \ref{asm:basic}, $v^{(K)}$ are uniformly bounded, say, by some $c_0<\infty$. By definition, for $0\leq k<K$ and $(x,z)\in\sX_N$,
$$
\alpha^{(K)} (s^{(K)}_k,x,z)=  \underset{a\in A(h^{(K)},x)}{\mathrm{arg\,min}} \ \Big\{\ell(x,z,a/h^{(K)})h^{(K)} + \sum_{y\in S(x)} a(y) \E[v^{(K)}(s^{(K)}_{k+1},y,Z'(x,z,\alpha^{(K)}(s^{(K)}_k)))]\Big\}.
$$
Further, define
$$
\eta^{(K)}(x,z) :=  \underset{a\in A(h^{(K)},x)}{\mathrm{arg\,min}} H(x,z,\vec{0},a)=  \underset{a\in A(h^{(K)},x)}{\mathrm{arg\,min}} \ell(x,z,a/h^{(K)}),\quad (x,z)\in\sX_N.
$$
By Lemma \ref{lem:convex-conjugates}, 
$$
\frac{1}{h^{(K)}} |\alpha^{(K)}(s^{(K)}_k,x,z)-\eta^{(K)}(x,z)|_1\leq \frac{1}{2\gamma} \, \sum_{y\in S(x)} \E|v^{(K)}(s^{(K)}_{k+1}, y, Z'(x,z,\alpha^{(K)}(s^{(K)}_k)))|\leq \frac{mc_0}{2\gamma}.
$$
By the definition of $\eta^{(K)}(x,z)$ as the minimizer of $H(x,z,\vec0,\cdot)$ over $A(h^{(K)},x)$, 
$
\ell(x,z,\eta^{(K)}(x,z)/h^{(K)}) \leq \ell (x,z,\bar a(x)),
$
for $\bar a(x)\in\R^d$ which can be chosen by $a(x)_y=0$ for $y\notin S(x)\cup\{x\}$ and $a(x)_y=\sigma^2$ for $y\in S(x)\setminus\{x\}$. Note that by Assumption \ref{asm:basic}, the running cost does not depend on $a(x)_x$. Now observe that $\ell(x,z,\cdot)$ is a strongly convex function on $\R^{\cX\setminus\{x\}}\cong\R^{d-1}$ and hence coercive. This implies boundedness of $(|\eta^{(K)}_{-x}(x,z)|_1/h^{(K)})_{K\geq1}$. Using $|a_{-x}-e_x|_1=2|a_{-x}|_1$ for $a\in\Sigma^{d-1}$, we see that $(|\eta^{(K)}(x,z)-e_x|_1/h^{(K)})_{K\geq1}$ is bounded by some $\tilde c_0<\infty$. Putting these two estimates together, we arrive at
$$
\sup_{K\geq K_0} \max_{0\leq k<K}\max_{(x,z)\in\sX_N}\frac{1}{h^{(K)}}|\alpha^{(K)}(s^{(K)}_k,x,z)-e_x|_1 \leq \frac{mc_0}{2\gamma}+\tilde c_0 =: c_1<\infty.
$$
Note that $(|\alpha^{(K)}(t,x,z)|_1/h^{(K)})_{K\geq K_0}$ is not a bounded sequence.

\emph{Equicontinuity.} Let us first observe that, by the same reasoning as in Lemma \ref{lem:cts-time-limit}, the uniform boundedness of $(\mathfrak{a}^{(K)}/h^{(K)})_{K\geq K_0}$ implies that
\begin{align*}
&\max_{0\leq k<K} \, \frac{1}{h^{(K)}}\,\P^{\alpha^{(K)}\otimes\alpha^{(K)}}[(X_{s^{(K)}_{k+1}},Z_{s^{(K)}_{k+1}})=(x',z')\,|\,(X_{s^{(K)}_{k}},Z_{s^{(K)}_{k}})=(x,z)]\\
&\qquad =\max_{0\leq k<K} \,   \frac{1}{h^{(K)}}\, \alpha^{(K)}_{x'}(s^{(K)}_k,x,z) \,\P[Z'(x,z,\alpha^{(K)}(s^{(K)}_k))=z']
\end{align*}
is bounded, say by some $c_2<\infty$, whenever $(x',z')\neq (x,z)$. Hence, for $0\leq k<K$ and $(x,z)\in\sX_N$,
\begin{align*}
&|v^{(K)}(s^{(K)}_{k+1},x,z)-v^{(K)}(s^{(K)}_k,x,z)|\\
&\quad \leq |\E^{\alpha^{(K)}\otimes\alpha^{(K)}}[v^{(K)}(s^{(K)}_{k+1},X_{s^{(K)}_{k+1}},Z_{s^{(K)}_{k+1}})]-v^{(K)}(s^{(K)}_k,x,z)|  \\
&\quad + |\E^{\alpha^{(K)}\otimes\alpha^{(K)}}[v^{(K)}(s^{(K)}_{k+1},X_{s^{(K)}_{k+1}},Z_{s^{(K)}_{k+1}})]-v^{(K)}(s^{(K)}_{k+1},x,z)| \\
&\quad= |\E^{\alpha^{(K)}\otimes\alpha^{(K)}}[v^{(K)}(s^{(K)}_{k+1},X_{s^{(K)}_{k+1}},Z_{s^{(K)}_{k+1}})]-v^{(K)}(s^{(K)}_k,x,z)| \\
&\quad+ \!\!\!\!\! \sum_{(x',z')\neq(x,z)} \!\!\!\!\!\!\! \P^{\alpha^{(K)}\otimes\alpha^{(K)}}[(X_{s^{(K)}_{k+1}},Z_{s^{(K)}_{k+1}})=(x',z')\,|\,(X_{s^{(K)}_{k}},Z_{s^{(K)}_{k}})=(x,z)] \, |v^{(K)}(s^{(K)}_{k+1},x',z')-v^{(K)}(s^{(K)}_{k+1},x,z)|\\
&\quad\leq |\ell(x,z,\alpha^{(K)}(s_k^{(K)},x,z)/h^{(K)})|h^{(K)} + 2c_0c_2\, h^{(K)} \leq (\sup_{D} |\ell(x,z,\cdot)|+2c_0c_2) \, h^{(K)} =: c_3 h^{(K)}.
\end{align*}
where $D:=\{(a_y)_{y\in\cX}\in\R^d:a_y\in[0,c_1/2] \ \text{for all} \ y\neq x\}$.
This implies equicontinuity of $(v^{(K)})_{K\geq K_0}$. We turn to  show the equicontinuity of $(\mathfrak{a}^{(K)})_{K\geq K_0}$. 
As in the proof in Proposition \ref{prop:finite-player-one-step}, by Lemma \ref{lem:convex-conjugates},
\begin{align*}
|\alpha^{(K)}(s^{(K)}_k,x,z)&\,-\alpha^{(K)}(s^{(K)}_{k+1},x,z)|_1\\
&\leq \frac{h^{(K)}}{2\gamma}\sum_{y\in S(x)}\E|v^{(K)}(s^{(K)}_{k+1},y,Z'(x,z,\alpha^{(K)}(s^{(K)}_k)))- v^{(K)}(s^{(K)}_{k+2},y,Z'(x,z,\alpha^{(K)}(s^{(K)}_{k+1})))|\\
&\leq \frac{h^{(K)}}{2\gamma}\sum_{y\in S(x)}\E|v^{(K)}(s^{(K)}_{k+1},y,Z'(x,z,\alpha^{(K)}(s^{(K)}_k)))- v^{(K)}(s^{(K)}_{k+1},y,Z'(x,z,\alpha^{(K)}(s^{(K)}_{k+1})))|\\
&+\,\frac{h^{(K)}}{2\gamma}\sum_{y\in S(x)}\E|v^{(K)}(s^{(K)}_{k+1},y,Z'(x,z,\alpha^{(K)}(s^{(K)}_{k+1})))- v^{(K)}(s^{(K)}_{k+2},y,Z'(x,z,\alpha^{(K)}(s^{(K)}_{k+1})))|\\
&\leq \frac{m L^{(K)}_{k+1} h^{(K)}}{2\gamma} \max_{(x,z)\in\sX_N}  |\alpha^{(K)}_{-x}(s^{(K)}_k,x,z)-\alpha^{(K)}_{-x}(s^{(K)}_{k+1},x,z)|_1 + \frac{c_3m}{2\gamma}(h^{(K)})^2,
\end{align*}
where $L^{(K)}_k:=\max_{x\in\cX}\,\Lip_1(v^{(K)}(s^{(K)}_k,x,\cdot))$.
Taking the maximum over $(x,z)\in\sX_N$, 
$$
\bigl(2\gamma- mL^{(K)}_{k+1}h^{(K)} \bigr)\max_{(x,z)\in\sX_N}  |\alpha^{(K)}(s^{(K)}_k,x,z)-\alpha^{(K)}(s^{(K)}_{k+1},x,z)|_1\leq  c_3m(h^{(K)})^2.
$$
Since $L^{(K)}_k$ is bounded uniformly in $(k,K)$, we see that indeed $(\alpha^{(K)}/h^{(K)})_{K\geq K_0}$ is equicontinuous. By definition of $\mathfrak{a}^{(K)}$, this shows equicontinuity of  $(\mathfrak{a}^{(K)}/h^{(K)})_{K\geq K_0}$.\\
 
\emph{Step 2.} By the Arzel\`{a}-Ascoli theorem, there exist continuous $(\tilde v^{(\infty)}, \tilde \alpha^{(\infty)})$ such that, along a subsequence, $(v^{(K)}, \mathfrak{a}^{(K)}/h_K)\to (\tilde v^{(\infty)}, \tilde \alpha^{(\infty)})$ uniformly on $[0,T]$ as $K\to\infty$. We replace the original sequence with this subsequence. \\

\emph{Step 3.} For $(x,z)\in\sX_N$, define 
\begin{align*}
\varphi^{(K)}_k(x,z)&:= H(x,z,E(x,z,v^{(K)}(s^{(K)}_{k+1}),\alpha^{(K)}(s^{(K)}_k)),\mathfrak{a}^{(K)}(s^{(K)}_k,x,z)) -v^{(K)}(s^{(K)}_{k+1},x,z),\quad 0\leq k<K\\
\varphi^{(\infty)}(s,x,z) &:=\cH(x,z,\Delta_x\tilde v^{(\infty)}(s,\cdot,z),\tilde \alpha^{(\infty)}(s,x,z))+\sL^{\tilde \alpha^{(\infty)}(s)}\tilde v^{(\infty)}(s)(x,z),\quad s\in[0,T]
\end{align*}
where $\sL^{\tilde\alpha^{(\infty)}}$ is as in \eqref{eq:generator} with $\beta=\tilde\alpha^{(\infty)}$.
We claim that
\be\label{eq:proof-cts-time-limit-1}
\lim_{K\to\infty}\ \max_{0\leq k<K} \ \Bigl|\frac{1}{h^{(K)}}\varphi^{(K)}_k(x,z)-\varphi^{(\infty)}(s^{(K)}_k,x,z)\Bigr| =0.
\ee
Indeed, write 
$
\varphi^{(K)}_k(x,z) = \ell(x,z,\mathfrak{a}^{(K)}(s^{(K)}_k,x,z)/h^{(K)})h^{(K)} + \mathrm{(I)}_k + \mathrm{(II)}_k,
$
where
\begin{align*}
\mathrm{(I)}_k &:=\sum_{y\in\cX} \mathfrak{a}^{(K)}_y(s^{(K)}_k,x,z)\, \E[v^{(K)}(s^{(K)}_{k+1},y,Z'(x,z,\alpha^{(K)}(s^{(K)}_k)))-v^{(K)}(s^{(K)}_{k+1},x,Z'(x,z, \alpha^{(K)}(s^{(K)}_k)))], \\
\mathrm{(II)}_k &:= \E[v^{(K)}(s^{(K)}_{k+1},x,Z'(x,z,\alpha^{(K)}(s^{(K)}_k)))]-v^{(K)}(s^{(K)}_{k+1},x,z).
\end{align*}
Using the fact that $\lim_{K\to\infty}Z'(x,z,\alpha^{(K)}(s^{(K)}_k)) = z$ in probability by Lemma \ref{lem:cts-time-limit} and the uniform convergence of $(v^{(K)},\mathfrak{a}^{(K)}/h^{(K)})$ to $(\tilde v^{(\infty)},\tilde \alpha^{(\infty)})$, 
$$
\lim_{K\to\infty} \ \max_{0\leq k< K} \ \Bigl|\frac{1}{h^{(K)}}\mathrm{(I)}_k -\sum_{y\in\cX} \tilde \alpha^{(\infty)}_y(s^{(K)}_k,x,z)\, (\tilde v^{(\infty)}(s^{(K)}_k,y,z)-\tilde v^{(\infty)}(s^{(K)}_k,x,z)) \Bigr| = 0.
$$
For the second term, we write 
$$
\mathrm{(II)}_k = \sum_{z'\in\Sigma^{d-1}_N} \P[Z'(x,z,\alpha^{(K)}(s^{(K)}_k))=z'] \, \big(v^{(K)}(s^{(K)}_{k+1},x,z')-v^{(K)}(s^{(K)}_{k+1},x,z)\big).
$$
We then use Lemma \ref{lem:cts-time-limit} and the uniform convergence of $(v^{(K)},\mathfrak{a}^{(K)}/h^{(K)})$ to $(\tilde v^{(\infty)},\tilde \alpha^{(\infty)})$ to deduce that
$$
\lim_{K\to\infty}\ \max_{0\leq k<K}\ \Bigl|\frac{1}{h^{(K)}}\mathrm{(II)}_k -\sL^{\tilde \alpha^{(\infty)}(s_k^{(K)})} \tilde v^{(\infty)}(s^{(K)}_k)(x,z)\Bigr| =0.
$$
This establishes \eqref{eq:proof-cts-time-limit-1}.\\

\emph{Step 4.} Note that, for every $0\leq k < K$, we have
\be\label{eq:proof-cts-time-limit-2}
-\big(v^{(K)}(s^{(K)}_{k+1},x,z)-v^{(K)}(s^{(K)}_k,x,z)\big) = \varphi^{(K)}_k(x,z),\quad (x,z)\in\sX_N.
\ee
Let $t\in(0,T)$, choose a sequence $t^{(K)}\in\{0,T/K,\ldots,T\}$ converging to $t$ and let $T^{(K)}=t^{(K)} K$. Summing over $0\leq k \leq T^{(K)}-1$ in \eqref{eq:proof-cts-time-limit-2}, we obtain $-(v^{(K)}(t^{(K)},x,z)-v^{(K)}(0,x,z))=\sum_{k=0}^{T^{(K)}-1} \varphi^{(K)}_k(x,z)$. Now,
\begin{align*}
\Bigl|\sum_{k=0}^{T^{(K)}-1} \varphi^{(K)}_k(x,z) - \int_0^t \varphi^{(\infty)}(s,x,z)\,\d s \Bigr| &\leq h^{(K)} \sum_{k=0}^{T^{(K)}-1} |\varphi^{(K)}_k(x,z)/h^{(K)}-\varphi^{(\infty)}(s^{(K)}_k,x,z)| \\
&+ \Bigl|\sum_{k=0}^{T^{(K)}-1} \varphi^{(\infty)}(s^{(K)}_k,x,z)h^{(K)}-\int_0^t \varphi^{(\infty)}(s,x,z)\,\d s \Bigr|.
\end{align*}
The first term converges to zero by Step 3.\ while the second term converges to zero by classical arguments. This proves that $\tilde v^{(\infty)}(t,x,z)-\tilde v^{(\infty)}(0,x,z)=-\int_0^t\varphi^{(\infty)}(s,x,z)\,\d s$. Hence, $\tilde v^{(\infty)}(\cdot,x,z)$ is continuously differentiable on $[0,T]$ with derivative $-\varphi^{(\infty)}(\cdot,x,z)$. 

\emph{Step 5.} We claim that
$$
(\tilde \alpha^{(\infty)}(t,x,z))_{y\neq x}=\underset{a\in\cR(\sigma^2,x)}{\mathrm{arg\,min}}\ \cH(x,z,\Delta_x\tilde v^{(\infty)}(t,\cdot,z),a),\quad (t,x,z)\in[0,T]\times\sX_N.
$$
Indeed, for every $0\leq k <K$, $(x,z)\in\sX_N$ and $a\in A(h^{(K)},x)$, we have
\begin{align*}
H(x,z,E(x,z,v^{(K)}(t^{(K)}+h^{(K)}),&\,\alpha^{(K)}(t^{(K)})),\mathfrak{a}^{(K)}(t^{(K)},x,z))  \\
&\leq H(x,z,E(x,z,v^{(K)}(t^{(K)}+h^{(K)}),\alpha^{(K)}(t^{(K)})),a).
\end{align*}
Adding $\E[v(t^{(K)}+h^{(K)},x,Z'(x,z,\alpha^{(K)}(t^{(K)})))]$ to both sides, dividing by $h^{(K)}$ and taking the limit $K\to\infty$ yields
$$
\cH(x,z,\Delta_x\tilde v^{(\infty)}(t,\cdot,z),\tilde \alpha^{(\infty)}(t,x,z))  \leq \cH(x,z,\Delta_x\tilde v^{(\infty)}(t,\cdot,z), a),\qquad (x,z)\in\sX_N,\, a\in\cR(\sigma^2,x).
$$

\emph{Step 6.} Steps 1--5 show that, given any subsequence $(K_n)_{n\geq1}$, we can find a further subsequence $(K_{n_p})_{p\geq1}$ such that $\{v^{(K_{n_p})}(\cdot,x,z):(x,z)\in\sX_N\}$ and $\{\mathfrak{a}^{(K_{n_p})}_y(\cdot,x,z) / h^{(K_{n_p})}:(x,z)\in\sX_N\}$ converge uniformly to some continuous functions $\tilde v^{(\infty)}= \{\tilde v^{(\infty)}(\cdot,x,z):(x,z)\in\sX_N\}$ and $\tilde \alpha^{(\infty)}=\{\tilde \alpha^{(\infty)}(\cdot,x,z):(x,z)\in\sX_N\}$, respectively, such that $(\tilde v^{(\infty)},\tilde a^{(\infty)})$ satisfy the continuous-time NLL equation \eqref{eq:cts-time-finite-player-NLL}. Since the latter has a unique classical solution $(v_{\mathrm{cts}},\alpha_{\mathrm{cts}})$, this proves that the entire sequence $(v^{(K)},\mathfrak{a}^{(K)}/h^{(K)})$ converges uniformly to $(v_{\mathrm{cts}},\alpha_{\mathrm{cts}})$.
\end{proof}

\appendix
\section{Appendix}

\subsection{Dynamic programming for Markov chains} \label{app:dynamic-progamming}

For $h>0$ and $K\geq1$, let $\bar \cT:=\{0,h,\ldots,Kh\}$ and $\cT:=\bar\cT\setminus\{Kh\}$.
Consider a state process $(X_t,Y_t)_{t\in\bar\cT}$ taking values in a discrete set $\cX\times\cY$, where $|\cX|=d_1<\infty$ and $|\cY|=d_2<\infty$. Fix a map $\beta:\cT\times\cX\times\cY\to\Sigma^{d_2-1}$. For any control $\alpha:\cT\times\cX\times\cY\to\Sigma^{d_1-1}$, let $\P^\alpha$ denote the probability measure under which $(X,Y)$ is a Markov chain with $\P^\alpha$-transition probabilities
$$
\P^\alpha[(X_{t+h},Y_{t+h})=(x',y')\,|\,(X_t,Y_t)=(x,y)] = \alpha(t,x,y)(x')\beta(t,x,y)(y').
$$
Let $\ell:\cT\times\cX\times\cY\times\Sigma^{d_1-1}\to\R$ and $g:\cX\times\cY\to\R$ be given (cost) functions. Introduce the value function
$$
v(t,x,y)=\inf_{\alpha:\,\cT\times\cX\times\cY\to\Sigma^{d_1-1}} \E^\alpha\Bigl[\sum_{t\leq s < T} \ell(s,X_s,Y_s,\alpha(s,X_s,Y_s))+g(X_s,Y_s)\,|\,(X_t,Y_t)=(x,y)\Bigr].
$$
Clearly, for the problem at time $t$, the values of $\alpha$ at previous times $s<t$ are irrelevant.

Define the Hamiltonian
$$
\sH(t,x,y,\varphi,a) = \ell(t,x,y,a)+\sum_{(x',y')\in\cX\times\cY}a(x')\beta(t,x,y)(y')\varphi(x',y').
$$
The following result is standard, see for example \cite{FS}.

\begin{Lem}\label{lem:bellman}
Let $w:\bar\cT\times\cX\times\cY\to\R$ be given and assume that the cost functions $\ell,g$ are continuous. Then $w=v$ if and only if $w$ solves the following dynamic programming equation, for $(t,x,y)\in\cT\times\cX\times\cY$,
$$
\begin{cases}
w(t,x,y) = \underset{a\in\Sigma^{d_1-1}}{\inf}\ \sH(t,x,y,w(t+h),a), \\
w(T,x,y) = g(x,y).
\end{cases}
$$
Moreover, a control $\alpha:\cT\times\cX\times\cY\to\Sigma^{d_1-1}$ is optimal for any initial point $(t,x,y)$ if and only if 
$$
\alpha(t,x,y)\in\underset{a\in\Sigma^{d_1-1}}{\mathrm{arg\,min}} \ \sH(t,x,y,v(t+h),a),
$$
for all $(t,x,y)\in\cT\times\cX\times\cY$. 
\end{Lem}

\begin{Rmk}
{\rm There are no issues with measurability since $\cT\times\cX\times\cY$ is a discrete set. Moreover, when the dynamics of $Y$ are deterministic, we may allow $\cY$ to be any set without changing the validity of Lemma \ref{lem:bellman}.
}
\end{Rmk}

\subsection{Proof of Lemma \ref{lem:convex-conjugates}} \label{app:conjugate-correspondence}

We follow the notation of \cite{FOMO}. Given a function $f:\R^d\to(-\infty,\infty]$, the domain of $f$ and of its subdifferential $\partial f$ are defined by
$$
\dom( f) := \{x\in\R^d: f(x)<\infty\}\quad \text{and}\quad \dom(\partial f) := \{x\in\R^d:\partial f(x)\neq\emptyset\}.
$$
Further, the indicator function of a set $C\subset\R^d$ is defined by 
$$
\delta_C:\R^d\to[0,\infty],\quad \delta_C(x):=
\begin{cases}
0 &x\in C,\\
\infty, &x\notin C.
\end{cases}
$$
Finally, the normal cone of a set $C$ is denoted by
$$
N_C(x) := 
\begin{cases}
\{n\in\R^d : n\cdot(c-x)\leq 0,\ \forall c\in C\}, &x\in C,\\
\emptyset, &x\notin C.
\end{cases}
$$

\emph{Step 1.} We first show that $\eta:=f+\delta_C$ is a  closed, proper and strongly convex function. Clearly, $\eta$ is proper. Closedness of $\eta$ follows from the closedness of $f$, for example using closedness of the sublevel sets \cite[Theorem 2.6]{FOMO}. Next, using $\partial\delta_C=N_C$, see \cite[Example 3.5]{FOMO}, and the sum rule of subdifferential calculus, see \cite[Theorem 3.40]{FOMO}, we have
$$
\partial \eta(x) = \partial f(x)+ N_C(x),\quad x\in\R^d,
$$ 
where the right side is understood in the sense of Minkowski. Hence $\dom(\partial\eta)=\dom(\partial f)\cap \dom(\partial\delta_C)=\dom(\partial f)\cap C$.
Strong convexity of $\eta$ means that, for $(x,x')\in\dom(\partial \eta)\times\dom(\eta)$ and $\phi\in\partial \eta(x)$,
\be\label{eq:app-convex-i}
\eta(x')\geq \eta(x)+\phi\cdot(x'-x)+\frac{\gamma}{2} |x'-x|_2^2,
\ee
see \cite[Theorem 5.24]{FOMO}. Let such $(x,x',\phi)$ be given. By the above observation, there exists a $g\in\partial f(x)$ and $n\in N_C(x)$ with $\phi=g+n$. Since $x,x'\in C$, by
\eqref{eq:lem-convex-conjugate},
\begin{align*}
\eta(x')=f(x')&\geq f(x)+ g\cdot(x'-x)+ \frac{\gamma}{2} |x'-x|_p|x'-x|_q\\
&\geq f(x)+ g\cdot(x'-x)+\frac{\gamma}{2} |x'-x|_2^2=\eta(x)+ g\cdot(x'-x)+\frac{\gamma}{2} |x'-x|_2^2
\end{align*}
by H\"{o}lder's inequality.  Using $n\in N_C(x)$ and $x'\in C$, we have $n\cdot(x'-x)\leq 0$. Hence,
$$
\eta(x')\geq  \eta(x)+g\cdot(x'-x)+n\cdot(x'-x)+ \frac{\gamma}{2} |x'-x|_2^2 = \eta(x)+ \phi\cdot(x'-x)+ \frac{\gamma}{2} |x'-x|_2^2 ,
$$
establishing \eqref{eq:app-convex-i}.

\emph{Step 2.} Let $y_1,y_2\in\R^d$ be given and set $v_i:=\psi(y_i)$ for $i=1,2$. We claim that there exist $g_i\in\partial f(v_i)$ and $n_i\in N_C(v_i)$ with  $y_i=g_i+n_i$ for $i=1,2$.
Indeed, note that we may write
$$
\psi(y)=\underset{x\in C}{\mathrm{arg\,max}} \{y\cdot x-f(x)\}=\underset{x\in \R^d}{\mathrm{arg\,max}} \{y\cdot x-\eta(x)\},\quad y\in\R^d.
$$
By Step 1, $\psi$ is well-defined as a single valued function, see \cite[Theorem 5.25 (a)]{FOMO}. By the conjugate subgradient theorem \cite[Theorem 4.20]{FOMO}, $v_i=\nabla \eta^*(y_i)$ for $i=1,2$.
Then, by \cite[Corollary 4.21]{FOMO} and the sum rule again,
$$
y_i\in\partial \eta(v_i) = \partial f(v_i)+ N_C(v_i),\quad i=1,2.
$$
The claim follows.

\emph{Step 3.} Let $r=p,q$ and let $r'$ be its conjugate.  We  claim that
\be \label{eq:app-convex-ii}
|v_1-v_2|_r\leq \frac{1}{\gamma}|y_1-y_2|_r.
\ee
If $v_1=v_2$, then the claim follows, so that we may assume $w:=v_2-v_1\neq0$.
By \eqref{eq:lem-convex-conjugate},
$$
f(v_2) \geq f(v_1)+ g_1\cdot w+ \frac{\gamma}{2} |w|_r |w|_{r'}\quad \text{and}\quad 
f(v_1) \geq f(v_2) - g_2\cdot w + \frac{\gamma}{2} |w|_r |w|_{r'}.
$$
Since $v_1,v_2\in C$, we have that 
$$
n_1\cdot w= n_1\cdot(v_2-v_1)\leq 0,\quad 
- n_2\cdot w = n_2\cdot(v_1-v_2)\leq 0.
$$
Hence, 
$$
f(v_2) \geq f(v_1)+ y_1\cdot w + \frac{\gamma}{2} |w|_r |w|_{r'}\quad \text{and}\quad 
f(v_1) \geq f(v_2) - y_2\cdot w + \frac{\gamma}{2} |w|_r |w|_{r'}.
$$
Adding these inequalities gives
$$
(y_2-y_1)\cdot w\geq \gamma |w|_r |w|_{r'}.
$$
By H\"{o}lder's inequality,
$$
\gamma |w|_r |w|_{r'}\leq |y_2-y_1|_r|w|_{r'},
$$
which, using $w\neq 0$, implies \eqref{eq:app-convex-ii}. \hfill $\Box$

{\small
\bibliographystyle{abbrvnat}
\bibliography{bibliography}
}
\end{document}